\documentclass[a4paper, 11pt]{amsart}
\usepackage{graphicx}
\usepackage{amsfonts}
\usepackage{indentfirst,bm,amsmath}
\usepackage{amsthm,amssymb}
\usepackage{mathrsfs}
\usepackage{hyperref}

\usepackage{subfig}

\newtheorem{theorem}{Theorem}[section]
\newtheorem{corollary}[theorem]{Corollary}
\newtheorem{lemma}[theorem]{Lemma}
\newtheorem{proposition}[theorem]{Proposition}
\newtheorem{definition}[theorem]{Definition}

\newtheorem{example}[theorem]{Example}

\begin{document}

	\title[]{Tropical Nevanlinna theory of several variables}
	
	\author[T. B. Cao \and J. H. Peng]{ Tingbin Cao\and Jiahu Peng}
	
	\address[Tingbin Cao and Jiahu Peng]{Department of Mathematics, Nanchang University, Nanchang city, Jiangxi 330031, P. R. China}
	\email{tbcao@ncu.edu.cn, pjh@email.ncu.edu.cn}
	
	\thanks{The first author is supported by the National Natural Science Foundation of China (\#12571082), the Jiangxi Natural Science Foundation (\#20232ACB201005) and the Shandong Natural Science Foundation (\#ZR2024MA024)}
	\date{}
	
	
	
	\begin{abstract} The main goal of this paper is to establish the higher-dimensional Nevanlinna theory in tropical geometry. We first develop a theory of tropical meromorphic functions ( holomorphic maps) in several variables, such as the proximity function, counting function and characteristic function, the first main theorem, higher-dimensional tropical versions of the logarithmic derivative lemmas. Based on this, for algebraically nondegenerate tropical holomorphic maps $f$ with subnormal growth from $\mathbb{R}^n$ into tropical projective space $\mathbb{TP}^{m}$ intersecting tropical hypersurfaces $\{V_{P_j}\}_{j=1}^{q}$ with  degree $d_{j},$ we then obtain the Second Main Theorem $$\|\,\,\, (q-M-1-\lambda)T_f(r) \leq \sum_{j=M+2}^q \tfrac{1}{d_j}N(r,1_{\mathbb{T}} \oslash P_j \circ f) + o(T_f(r)),$$ where  $d=lcd(d_{1}, \ldots, d_{q})$ and $M=(_d^{m+d})-1.$ 	\end{abstract}
	\maketitle
	\tableofcontents

	\section{Introduction and main results}\label{Sec-1}
	
	Nevanlinna theory \cite{hayman, nevanlinna}, developed by Rolf Nevanlinna in the 1920s, is a fundamental framework in complex analysis that examines the value distribution of meromorphic functions. Nevanlinna theory provides a quantitative description of how often a meromorphic function takes specific values, extending the Picard theorem on omitted values. The theory focuses on two main components: the \emph{characteristic function} $T(r,f)$, which measures the growth of $f$, and the \emph{counting function} $N(r,a,f)$, which counts the $a$-points of $f$. The First Main Theorem establishes a fundamental relation between these quantities, while the deeper Second Main Theorem provides inequalities governing the distribution of values. Its higher-dimensional generalizations investigate holomorphic maps $f \colon \mathbb{C}^{n} \to X$ into complex manifolds, with profound applications in complex geometry and Diophantine approximation (See, for example, \cite{griffiths, ru, stoll2}).\par
	
	Tropical geometry \cite{itenberg-mikhalki-shustin, maclagan, mikhalkin}, also referred to as max-plus algebra, has emerged as a powerful framework connecting algebraic geometry with combinatorial mathematics. The tropical semiring $\mathbb{T} := \mathbb{R}\cup\{-\infty\}$ is equipped with tropical operations:
	\begin{align*}
		x \oplus y := \max(x,y) \,\,\, \mbox{and}\,\,\, x \otimes y := x + y,
	\end{align*}
	We also use the notations $x\oslash y:=\frac{x}{y}\oslash=x-y$ and $x^{\otimes a}:=ax$ for $a\in \mathbb{R}.$ The set $\mathbb{T}$ endowed with tropical arithmetic operators is called the set of tropical numbers. These operations lead to a completely different algebraic structure compared to classical complex analysis, yet remarkably parallel results can be established.
	In 2009 Halburd and Southall \cite{tropnevan} first established the tropical Nevanlinna theory for continuous piecewise integer linear real functions in the one-dimensional tropical affine space $\mathbb{R}$ (called tropical meromorphic functions). They first proposed the tropical versions of the Poisson-Jensen theorem, the first main theorem, the logarithmic derivative lemma, and the Clunie theorem for tropical meromorphic functions with finite order. In 2011, Laine and Tohge \cite{laine-tohge} considered extended the definition of tropical meromrophic functions with arbitrary real slopes and established the tropical Nevanlinna's second main theorem for tropical meromorphic functions with hyperorder strictly less than one. In this case the multiplicities of poles (respectively, roots) may be arbitrary real numbers instead of being integers (respectively, rational numbers) which is in certain respects fundamentally different from the counterparts in the classical meromorphic functions of one complex variable. In 2016, Korhonen and Tohge \cite{korhonen-tohge-2016} extended the tropical Nevanlinna theory to a higher-dimensional tropical projective space $\mathbb{T}\mathbb{P}^{m}$  for tropical holomorphic curves and obtained a tropical analogue of Cartan's second main theorem. Recently, the present first author and Zheng \cite{hypersurfaces} extended the tropical Nevanlinna theory to the case of tropical hypersurfaces in tropical projective spaces in which the growth of order is improved to the subnormal growth (or called minimal hypertype),$  i. e. , \limsup_{r\rightarrow\infty}\frac{\log T_{f}(r)}{r}=0,$ due to an improvement of the tropical logarithmic derivative lemma. This is an emerging field that combines classical value distribution theory with tropical geometry. However, up to now, there are no any results on tropical Nevanlinna theory in several variables. The key point is how to describe the properties of Nevanlinna's characterstic functions (including proximity function and counting function) for tropical meromorphic functions in several variables.  \par
	
	The main purpose of this paper is to develop a comprehensive theory of tropical meromorphic functions in $\mathbb{R}^n$ and then establish some fundamental results in the tropical Nevanlinna theory for higher dimensions,  by combining the classical value distribution theory with tropical geometry.\par

    \subsection{Tropical meromorphic functions in several variables and Nevanlinna theory}	In Section \ref{Sec-2}, we will introduce tropical meromorphic functions $f:\mathbb{R}^n \rightarrow \mathbb{T}$ defined through tropical rational operations:
	\[ f(x) = \left(\bigoplus_{i=0}^{+\infty} a_i \otimes x^{\otimes \mathbf{m}_i}\right) \oslash \left(\bigoplus_{j=0}^{+\infty} b_j \otimes x^{\otimes \mathbf{n}_j}\right), \]
	where $x = (x_1,\ldots,x_n) \in \mathbb{R}^n$ and $\mathbf{m}_i, \mathbf{n}_j \in \mathbb{R}^n$ are exponent vectors. One can see that these functions are piecewise linear with polyhedral complexes as their natural domains. There are some properties of tropical meromorphic functions as follows:\par
	\begin{itemize}
		\item Global continuity and piecewise linearity in all variables
		\item Polyhedral complex structure (finite or infinite unions of concave/convex polyhedral cells)
		\item Real-valued gradients corresponding to exponent vectors $\mathbf{m}_i$ and $\mathbf{n}_j$
	\end{itemize}
	We will classify all points of a tropical meromorphic function as (Definition \ref{jf}):\par
	\begin{itemize}
		\item Tropical smooth points: $J_f(x;\varphi) = 0$ for all directions $\varphi$
		\item Tropical poles: $J_f(x;\varphi_0) < 0$ for some $\varphi_0$
		\item Tropical roots: all other non-smooth points
	\end{itemize}
	where $J_f(x;\varphi)$ captures the so-called $\varphi$-positive directional derivative operator (Definition \ref{positive dd}).\par
	
We will develop the basic definitions and notations of the tropical Nevanlinna theory in $\mathbb{R}^n$ in Section \ref{Sec-3}, which closely generalizes the one-dimensional case studied by Halburd-Southall \cite{tropnevan} and Laine and Tohge \cite{laine-tohge} (see also \cite{book}). The fundamental components are:\par
	
	\begin{itemize}
		\item the \emph{proximity function} (Definition \ref{proxi}):
		\[ m(r,f) := \frac{1}{\omega_n}\int_{S^{n-1}(1)} f^+(r\theta)d\sigma(\theta); \]
		\item the \emph{counting function} for poles (Definition \ref{ctf}):
		\[ N(r,f) := \frac{1}{2}\int_0^r n(t,f)dt, \] where
		\begin{align*}
			n(t,f)&=\frac{1}{\omega_n}\int_{S^{n-1}(1)} n(t,f_{\theta})\,d\sigma(\theta);
		\end{align*}		\item the \emph{characteristic function}:
		\[ T(r,f) := m(r,f) + N(r,f). \]
	\end{itemize}
	There are identical concepts of the proximity function and the counting function, see Lemma \ref{L-3.7} and Lemma \ref{L-3.9}, respectively, by the Dirac function. The higher-dimensional tropical Jensen formula (Lemma \ref{fmtn})
	\[ T(r,f) - T(r,1_{\mathbb{T}} \oslash f) = f(0), \]
	leads to the First Main Theorem for tropical meromorphic functions (Theorem \ref{T3.15})
		\[
		T(r, \mathbf{1}_{\mathbb{T}} \oslash (f \oplus a)) = T(r, f) + O(1),
		\] where $a\in\mathbb{T}$ satisfying $a < L_f := \inf \{f(b): b \text{ is a pole of } f\}.$ \par
        
	In Section \ref{Sec-4}, we make use of the Dirac function to obtain a useful lemma (Lemma \ref{stolllemma}) that is a real analogous result of a Biancofiore-Stoll's lemma \cite[Lemma 3.1]{stoll1} (see also \cite[A5.1.1]{ru}), and then establish tropical versions of the logarithmic derivative lemma:\par
	
	\begin{itemize}
		\item For difference operators (Theorem \ref{loga}):
		\[ m(r,f(x+c)\oslash f(x)) = o(T(r,f)) \]
		holds for a tropical meromorphic function with submormal growth, where \( r \) runs to infinity outside of a set \( E \) of zero upper density measure.
		\item For $q$-difference operators (Theorem \ref{qloga}):
		\[ m(r,f(qx)\oslash f(x)) = o(T(r,f)) \] 	holds for a tropical meromorphic function with zero order,  and for all $r$ on a set of logarithmic density $1.$
	\end{itemize}
	
\subsection{Tropical holomorphic maps and Nevanlinna theory}
Define the tropical projective space as
	$\mathbb{TP}^{m}=\mathbb{R}^{m+1}/\mathbb{R}(1, 1, \ldots, 1).$ In Section \ref{Sec-5}, we introduce the tropical holomorphic maps $f:=[f_{0}: f_{1}:\cdots: f_{m}]:\mathbb{R}^{n}\rightarrow \mathbb{TP}^{m},$ where $f_{0}, f_{1}, \ldots, f_{m}$ are tropical entire functions in $\mathbb{R}^{n}$ and do not have any roots which are common to all of them. Denote $\mathbf{f}=(f_{0}, f_{1}, \ldots, f_{m}):\mathbb{R}^n\rightarrow \mathbb{T}^{m+1}.$ Then the map $\mathbf{f}$ is called a reduced representation of the tropical holomorphic map $f$ in $\mathbb{TP}^{m}.$ The tropical Cartan's characteristic function of \( f \) is defined by
		\[
		T_{\mathbf{f}}(r) := \frac{1}{\omega_n} \int_{S^{n-1}(1)} \|f(r\theta)\|d\sigma(\theta)  - \|f(0)\|,
		\]
		where $\|f(r\theta)\| = \max \bigl\{ f_0(r\theta), \ldots, f_m(r\theta) \bigr\},$ $r\theta=x\in \mathbb{R}^n,$ $\theta\in S^{n-1}(1).$ The first main theorem for a tropical holomorphic map $f$ intersecting tropical hypersurface $V_p$ (Theorem \ref{T4}) states that 
		if $f(\mathbb{R}^n)\not\subset V_{P},$ then we have
		\begin{eqnarray*}m_{f}(r, V_{P})+N(r, 1_{\mathbb{T}}\oslash P(f))=d T_{f}(r)+O(1).\end{eqnarray*}\par

The above tropical logarithmic derivative lemmas enable us to obtain the Second Main Theorems for tropical hypersurfaces (Theorem \ref{SMT} and Theorem \ref{qsmt}), which for algebraically nondegenerate maps $f:\mathbb{R}^n \rightarrow \mathbb{TP}^m$ with subnormal growth intersecting tropical hypersurfaces $\{V_{P_j}\}_{j=1}^{q}$ with degree $d_j$ states:
	$$\|\,\,\, (q-M-1-\lambda)T_f(r) \leq \sum_{j=M+2}^q \tfrac{1}{d_j}N(r,1_{\mathbb{T}} \oslash P_j \circ f) + o(T_f(r)),$$
	where  $d=lcd(d_{1}, \ldots, d_{q}),$  $M = \binom{m+d}{d}-1$ and $\lambda:=ddg (\{P_{M+2}\circ f, \ldots, P_{q}\circ f\})$ measures degeneracy (see in Subsection \ref{SS2.3}).\par
	
Finally, in Section \ref{Sec-6} we present a growth-free Second Main Theorem (Theorem \ref{C1.4}) for a tropical holomorphic map intersecting a complete tropical polynomials (i.e., all coefficients are finite real values):
	\[ T_f(r) = \frac{1}{d}N\left(r,\frac{1_{\mathbb{T}}}{P\circ f}\oslash\right) + O(1), \]
	which generalizes  Halonen-Korhonen-Filipuk's results in \cite{8}, and in particular whenever $\mathbb{TP}^1$ this recovers the relationship:
	\[ T(r,f) = N\left(r,\frac{1_{\mathbb{T}}}{f\oplus a}\oslash\right) - N(r,f\oplus a) + N(r,f) + O(1) \]
	for any $a\in\mathbb{R}.$ \par

In a word, our work in this paper strongly extends previous results by Laine and Tohge \cite{laine-tohge} , Korhonen-Tohge  \cite{korhonen-tohge-2016} and Cao-Zheng \cite{hypersurfaces} to the case of several variables, while many new techniques are introduced.
	
\section{Tropical meromorphic functions in $\mathbb{R}^{n}$}\label{Sec-2}
	
	\subsection{Tropical algebra}
	Tropical operations (max-plus) are defined for two real variables $x, y\in\mathbb{T}:=\mathbb{R}\cup\{-\infty\},$
	\[
	\begin{aligned}
		x \oplus y &:= \max(x, y), \\
		x \otimes y &:= x + y, \\
		x \oslash y &:= x - y.
	\end{aligned}
	\] Obviously, the tropical additional unit $0_{\mathbb{T}}:=-\infty$ and the tropical  multiplication unit $1_{\mathbb{T}}:=0.$ \par

	\subsection{Tropical meromorphic functions in $\mathbb{R}^{n}$}
	Here, we introduce the concept of tropical meromorphic functions in higher dimensions as follows.
	
	\begin{definition}[Tropical meromorphic function] A tropical meromorphic function $f: \mathbb{R}^n \to \mathbb{T}$ is defined in the tropical algebra (max-plus) as follows:
		\begin{eqnarray}\label{mer}
			f(x) = \left( \bigoplus_{i=0}^{+\infty} a_i \otimes x^{\otimes \mathbf{m}_i} \right) \oslash \left( \bigoplus_{j=0}^{+\infty} b_j \otimes x^{\otimes \mathbf{n}_j} \right),
		\end{eqnarray}
		where:
		\begin{itemize}
			\item \(x = (x_1, x_2, \ldots, x_n) \in \mathbb{R}^n\);
			\item The exponential operation is given by:
			\[
			x^{\otimes \mathbf{m}_i} := m_{i1} x_1 + m_{i2} x_2 + \cdots + m_{in} x_n,
			\]
			where \(\mathbf{m}_i = (m_{i1}, m_{i2}, \ldots, m_{in}) \in \mathbb{R}^n\);
			\item All coefficients $a_i, b_j \in \mathbb{T}.$
	\end{itemize}\end{definition}
	
	By translating the max-plus operations into standard arithmetic operations, \eqref{mer} can be expressed as:
	\begin{eqnarray*}
		f(x) = \max_{0 \leq i \leq \infty} \left\{ a_i + m_{i1}x_1 + \cdots + m_{in}x_n \right\} -  \max_{0 \leq j \leq \infty} \left\{ b_j + n_{i1}x_1 + \cdots + n_{in}x_n \right\}.
	\end{eqnarray*}

	One can see that a tropical meromorphic function $f$ satisfies the following properties:
	\begin{itemize}
		\item $f$ is a globally continuous piecewise linear function in several variables $x$.
		In particular, for each $i \in \{1,\dots,n\}$, the univariate slice
		$x_i \mapsto f(x_1,\dots,x_i,\dots,x_n)$ is also a continuous piecewise linear function.
		
		\item The non-smooth locus $\Sigma_f$ of $f$ is an $(n-1)$-dimensional polyhedral complex,
		consisting of convex polyhedral cells glued together along common faces.
		
		\item Each linear region (cell) of $f$ has a constant real-valued gradient vector, corresponding to the exponent vector
		$\mathbf{m}_i$ or $\mathbf{n}_j$ of the dominating tropical monomial in that region.
	\end{itemize}
	
	\begin{example}[Tropical polynomials] A tropical polynomial $P: \mathbb{R}^n \to \mathbb{T}$ is defined as:
		\begin{eqnarray*}\label{tro mer}
			P(x) = \bigoplus_{i=0}^{k} a_i \otimes x^{\otimes \mathbf{m}_i}.
		\end{eqnarray*} The roots of $P(x)$ consist by all points where the maximum is achieved by at least two terms, i. e.,  $$\Sigma_P := \{ x \in \mathbb{R}^n :  \max_{i=0}^{k}\{a_i + m_i x\} \,\, \mbox{is attained by at least two terms}\}.$$
	\end{example}
	
	\begin{example}[Tropical rational functions] A tropical rational function $f: \mathbb{R}^n \to \mathbb{T}$ is defined in the tropical algebra (max-plus) as:
		\begin{eqnarray}\label{tro rat}
			f(x) = \left( \bigoplus_{i=0}^{p} a_i \otimes x^{\otimes \mathbf{m}_i} \right) \oslash \left( \bigoplus_{j=0}^{q} b_j \otimes x^{\otimes \mathbf{n}_j} \right),
		\end{eqnarray}
	\end{example}
	
	\begin{example}[Tropical entire functions] A tropical entire function $f: \mathbb{R}^n \to \mathbb{T}$ is defined in the tropical algebra (max-plus) as:
		\begin{eqnarray*}\label{tro mer}
			f(x) = \bigoplus_{i=0}^{+\infty} a_i \otimes x^{\otimes \mathbf{m}_i}
		\end{eqnarray*}which is a polyhedral complex: a finite or infinite union of convex polyhedral cells glued together along faces, i.e., has no poles.
	\end{example}

	\subsection{Tropical roots and poles for meromorphic functions}

	For all $ x \in \mathbb{R}^n $, we express it as $ x = t\theta $, where:
	\begin{itemize}
		\item if $ x \neq 0,$ then $ t = \|x\| >0 $ and $ \theta = \frac{x}{\|x\|} \in S^{n-1} (1):=\{y\in\mathbb{R}^{n}: \|y\|=1\} ;$
		\item if $ x = 0 ,$ then $ t = 0 $ and there are infinitely many directions.
	\end{itemize}
	For each fixed $ \theta \in S^{n-1}(1) $, define $ f_\theta(t) := f(t\theta) $ for $ t \in [0, \infty) $. Note that $ f_\theta(0) = f(0) $ is independent of $ \theta $.
	Then for each fixed $\theta$, the function $$f_{\theta}(t):=f(t\theta) $$ can be regarded as a one-dimensional tropical meromorphic function for the variable $t$. Throughout this paper, we will adopt this idea to extend the tropical Nevalinna theory from one variable to several variables.\par
	
	It is known  that for the \(n\)-dimensional Euclidean space $\mathbb{R}^{n},$ the surface area $\omega_n$ of the unit sphere $S^{n-1}(1)$ is given by \cite{rudin}
	\[
	\omega_n := \frac{2 \pi^{\frac{n}{2}}}{\Gamma\left(\frac{n}{2}\right)},
	\] where \(\Gamma\) is the real Gamma  function, which satisfies \(\Gamma(n) = (n-1)!\). It is known that  $$\int_{S^{n-1}(1)}  d\sigma(\varphi)=\omega_n,$$
		where  \( d\sigma(\varphi) \) is the standard surface measure.	\par
	
	\begin{definition}\label{positive dd}
		Let  \( f: \mathbb{R}^n \to \mathbb{T} \) be a tropical meromorphic function.  For any point $x\in\mathbb{R}^{n}$, and a direction $\varphi\in S^{n-1}(1),$  we define the $\varphi$-positive directional derivative operator by	\begin{align*}
			\partial_{\varphi}^{+}f(x) &:= \lim_{\substack{h \to 0^+ \\ h > 0}} \frac{f(x + h\varphi) - f(x)}{h}.
		\end{align*}
		Set
		\begin{equation*}
			J_f(x; \varphi) := \partial_{\varphi}^{+}f(x) + \partial_{-\varphi}^{+}f(x)
		\end{equation*}
		and
		\begin{equation*}
			\overline{J_f}(x) := \frac{1}{\omega_n} \int_{S^{n-1}(1)} J_f(x; \varphi)  d\sigma(\varphi),
		\end{equation*}
		where  \( d\sigma(\varphi) \) is the standard surface measure.	
	\end{definition}	
	
	According to the definition, direct verification shows the following result:
	\begin{equation} \label{J=}
		J_f(x; \varphi)=J_f(x; -\varphi).
	\end{equation}
	
	\begin{example}
		For a tropical meromorphic function in  one dimension (i.e., $\mathbb{R}^1$), we can deduce that  $S^0(1)= \{-1, 1\},$
		$\omega_1 = 2.$ Then for \(\varphi\in \{-1, 1\}\), without loss of generality, we have
		\begin{align*}
			\partial_{1}^{+}f(x) &= \lim_{\substack{h \to 0^+ \\ h > 0}} \frac{f(x + h) - f(x)}{h} = f'_+(x), \\
			\partial_{-1}^{+}f(x) &= \lim_{\substack{h \to 0^+ \\ h > 0}} \frac{f(x - h) - f(x)}{h} = -f'_-(x), \\
			J_f(x; 1) &=J_f(x; -1)= f'_+(x) - f'_-(x),
		\end{align*}
		and \begin{align*}
			\overline{J_f}(x)&=\frac{1}{\omega_1}\int_{S^{0}(1)} J_f(x;\varphi) d\sigma(\varphi) \\&= \frac{1}{2} \sum_{\varphi \in \{-1,1\}} J_f(x, \varphi)\\
			&= J_f(x; 1) = J_f(x; -1).
		\end{align*}
	\end{example}

	For a tropical meromorphic function $f: \mathbb{R}^n \to \mathbb{T},$ all points where the function is not differentiable are called tropical roots or poles. The  details will be shown as follows.\par
	
	\begin{definition}\label{jf}
		Let $f: \mathbb{R}^n \to \mathbb{T}$ be a tropical meromorphic function. We classify $x\in\mathbb{R}^n $ as follows:
		
		(i). (tropical) smooth point if
		\[
		J_f(x; \varphi) = 0 \quad \text{for every direction } \varphi;
		\]
		
		(ii). tropical pole if there exists at least one $\varphi_0,$ such that
		\[
		J_f(x; \varphi_0) < 0,
		\] and its multiplicity is defined by  \begin{align*}
			\nu_f(x):=\frac{1}{\omega_n}\int_{\{\varphi_0 \in S^{n-1}(1): J_f(x; \varphi_0)<0\}} |J_f(x; \varphi_0)| d\sigma(\varphi_0);
		\end{align*}
		
		(iii). tropical root if	else, and its the multiplicity \begin{align*}
			\nu_{1_{\mathbf{T}} \oslash f}(x):=\frac{1}{\omega_n}\int_{\{\varphi_0 \in S^{n-1}(1): J_f(x; \varphi_0)>0\}} |J_f(x; \varphi_0)| d\sigma(\varphi_0).
		\end{align*}
	\end{definition}
	
	\begin{figure}[htbp]
		\centering
		\includegraphics[width=8cm]{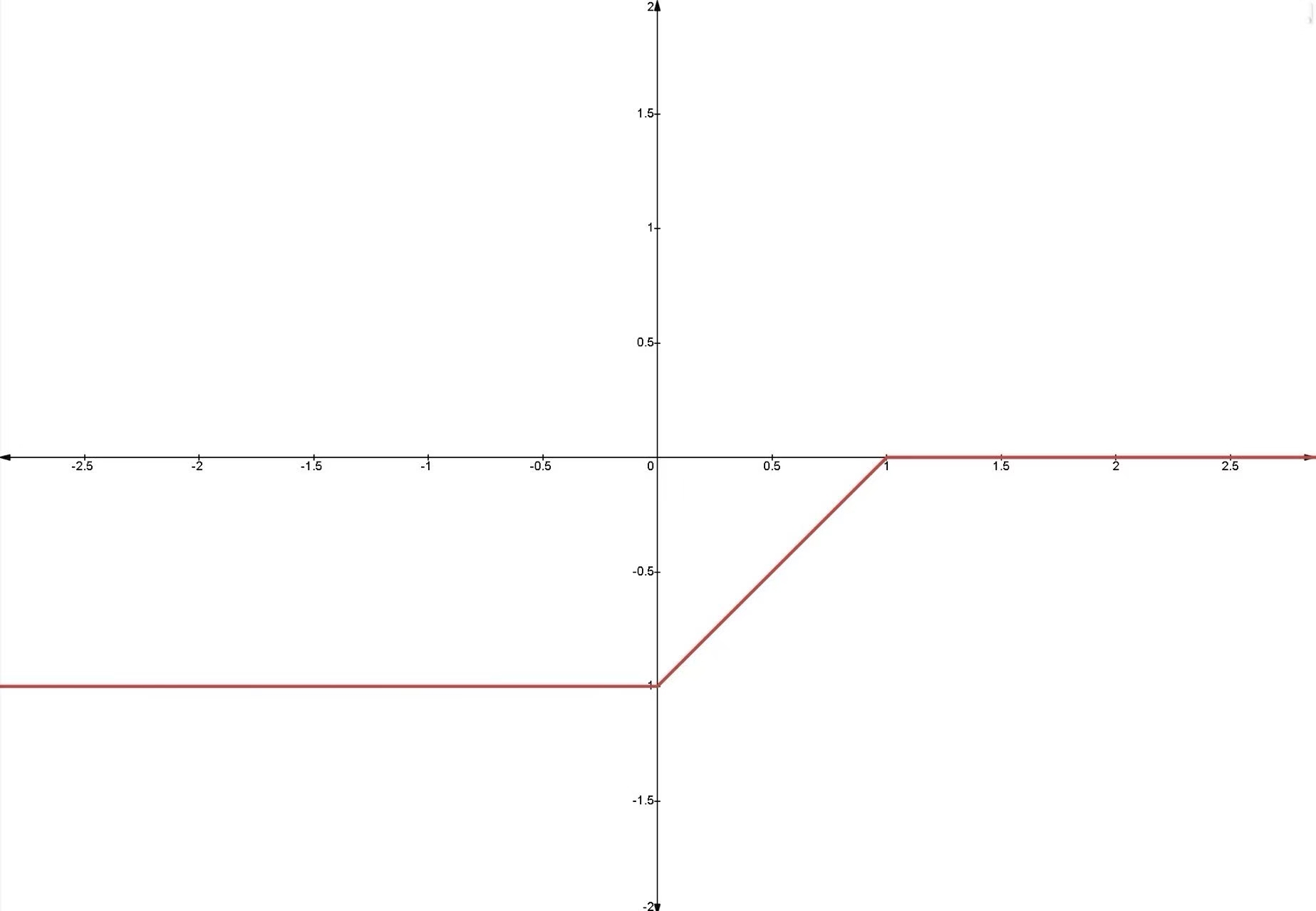}
		\caption{ $ 	f(x) = (x\oplus 0)\oslash (x\oplus 1)$}
		\label{F00}
	\end{figure}

	\begin{example}
		Consider the tropical meromorphic function of one variable
		\[
		f(x) = (x\oplus 0)\oslash (x\oplus 1)= \max(x, 0) - \max(x, 1).
		\]
		It is easy to see that (see Figure \ref{F00})
		\[
		f(x) =
		\begin{cases}
			-1 & \text{if } x \leq 0, \\
			x - 1 & \text{if } 0 < x \leq 1, \\
			0 & \text{if } x > 1.
		\end{cases}
		\]
		
		Then for $\varphi=1,$ one can deduce that
		$J_f(0, \varphi)=J_f(0, -\varphi)=f'_+(0)-f'_-(0) = 1>0$ and
		$J_f(1, \varphi)=J_f(1, -\varphi)=f'_+(1) - f'_-(1) = -1<0.$
		Hence, the points $x=0$ and $x=1$ are a tropical root and tropical pole of the function $f$ with multiplicity one, respectively.	
	\end{example}
	
	\begin{figure}[htbp]
		\centering
		\includegraphics[width=8cm]{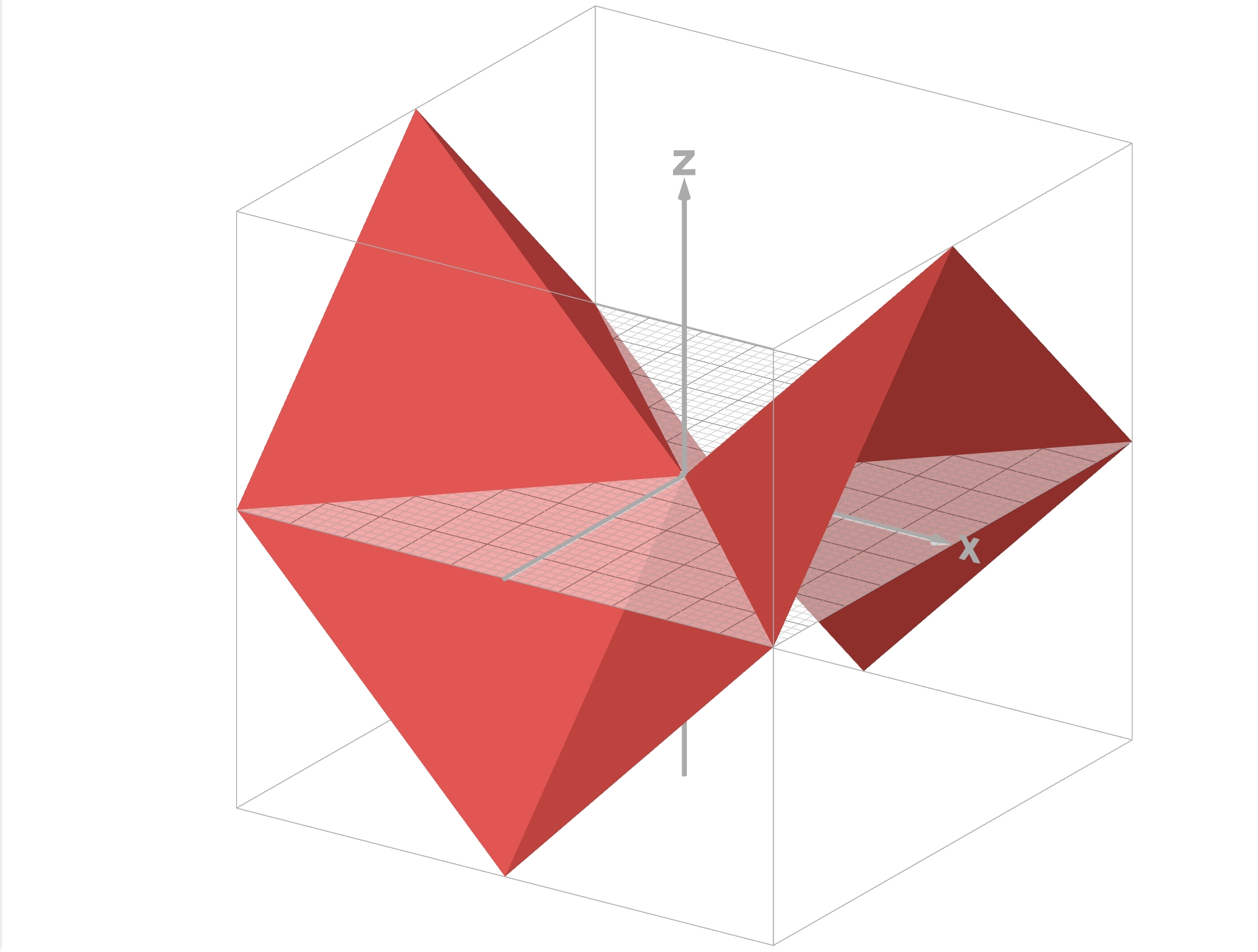}\\
		\caption{$z=|x| - |y|$}
		\label{F2}
	\end{figure}
	
	\begin{example}	\label{EX} For two-dimensional case, we define a tropical meromorphic function (see Figure \ref{F2}):
		\begin{eqnarray*}
			z=f(x,y) := ( x \oplus (1_{\mathbb{T}}\oslash x)) \oslash (y\oplus (1_{\mathbb{T}}\oslash y) ) =|x| - |y|.
		\end{eqnarray*}
		For any direction $\varphi= (\cos\phi, \sin\phi),$  it is easy to deduce that
		\[
		J_f((0,0); \varphi) = 2(|\cos\phi| - |\sin\phi|).
		\]
		If  $\varphi= (1, 0)$ on $x$-axis (see the left hand of Figure \ref{F3}), then
		$J_f((0,0);(1,0))= 2>0;$ and if $\varphi = (0,1)$ on $y$-axis (see the right hand of Figure \ref{F3}), then $J_f((0,0);(0,1)) = -2<0.$
		Direct computation yields \begin{align*}
			\nu_f(x):&=\frac{1}{2\pi}\int_{0}^{2\pi} |2(|\cos\phi| - |\sin\phi|)| d\phi \,\,\, (\mbox{for all}\,\, |\cos\phi|<|\sin\phi|)\\
			&=\frac{4\sqrt{2}-4}{\pi}.
		\end{align*} Hence, the  point $(0,0)$ is tropical pole of $f$ with multiplicity $\frac{4\sqrt{2}-4}{\pi}.$
	\end{example}

	\begin{figure}
		\centering
		\subfloat{\label{Figure1}}
		\includegraphics[width=0.45\linewidth]{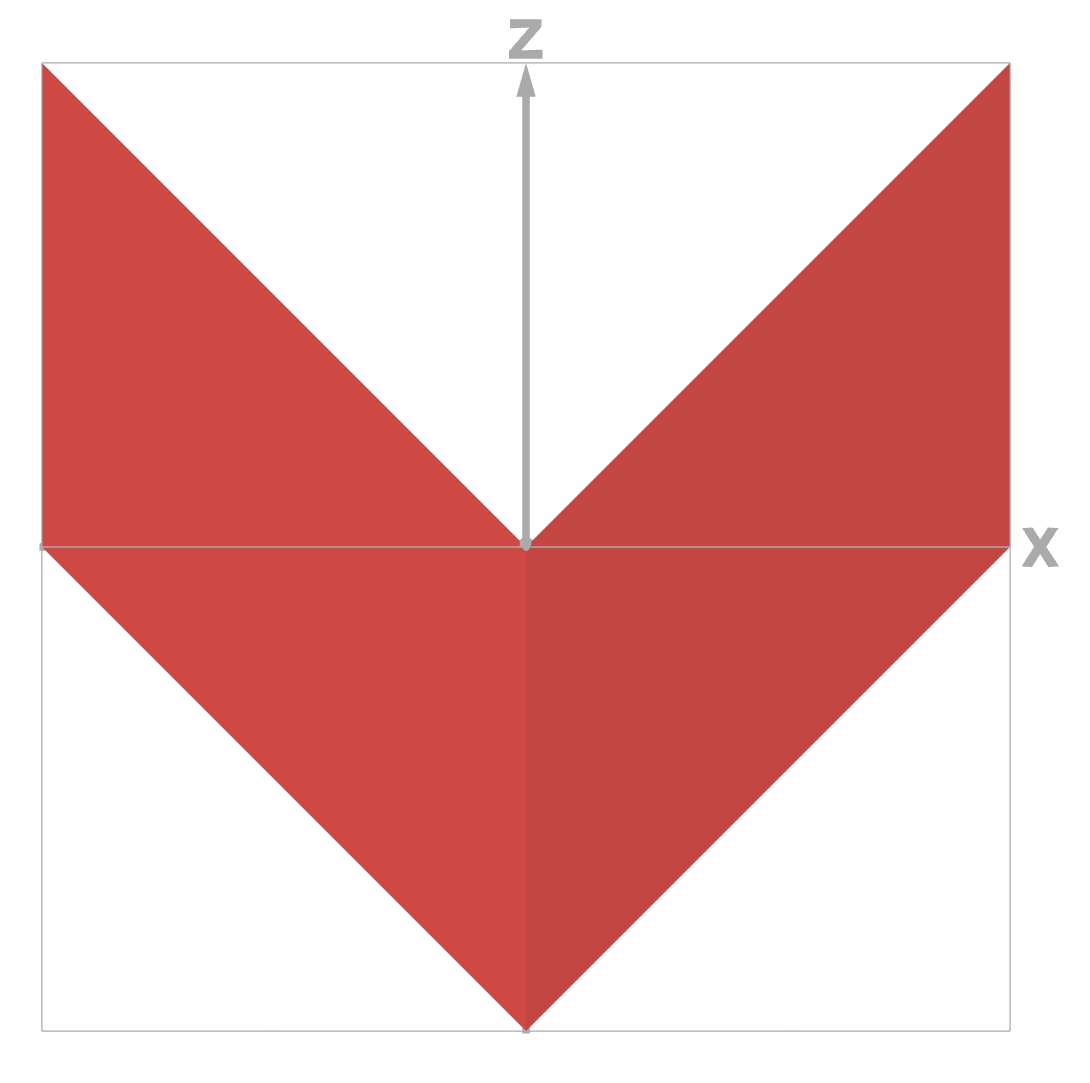}
		\quad
		\subfloat{\label{Figure2}}
		\includegraphics[width=0.45\linewidth]{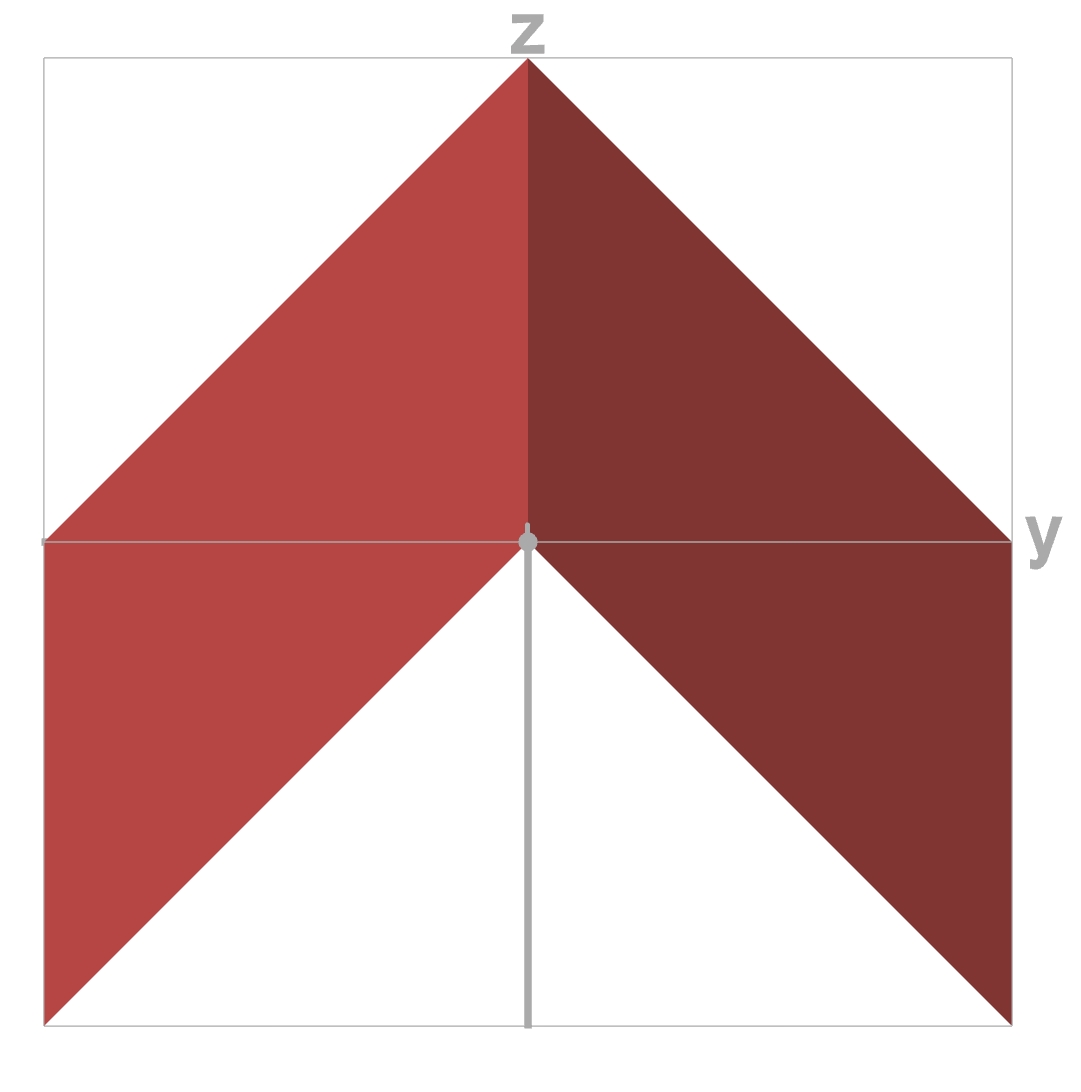}
		\caption{Example \ref{EX}}
		\label{F3}
	\end{figure}

	\section{Tropical first main theorem in higher dimensions}\label{Sec-3}
	\subsection{Basis on the Dirac functions}
	Firstly, we recall a basic result on the Dirac functions as follows.\par
	
	\begin{lemma}
		\cite[Ch.1, Sec.1]{dlk1}\label{L-1}
		The fundamental property of the Dirac delta distribution in one dimension:
		\begin{align}\label{E-1}
			\int_{\mathbb{R}} \delta(x - a) \phi(x)  dx = \phi(a);
	\end{align}	\end{lemma}
	
	By Lemma \ref{L-1}, we get a result which will play a key role in describing the basics of the tropical Nevanlinna theory in higher dimensions. \par
	
	\begin{lemma}\label{dlkn}
		\begin{align}\label{E2}
			\int_{\mathbb{R}^n} \delta(\|x\| - r) F(x)  dx = r^{n-1} \int_{S^{n-1}(1)} F(r\theta)  d\sigma(\theta),
		\end{align}
		where $x=r\theta$ and $  d\sigma(\theta)$is surface element.
	\end{lemma}
	
	\begin{proof}We perform a change of variables to spherical coordinates, setting $x = \rho \theta$, where $\rho \in [0, \infty)$ and $\theta \in S^{n-1}(1)$. The volume element in $\mathbb{R}^n$ becomes $dx = \rho^{n-1} d\rho\, d\sigma(\theta)$. Substituting into the integral gives
		\[
		\int_{\mathbb{R}^n} \delta(\|x\| - r)\, F(x)\, dx
		= \int_0^\infty \int_{S^{n-1}(1)} \delta(\rho - r)\, F(\rho \theta)\, \rho^{n-1} \, d\sigma(\theta)\, d\rho.
		\] Using the standard property of the Dirac delta function \eqref{E-1},
		we evaluate the $\rho$-integral:
		\[
		\int_0^\infty \delta(\rho - r)\, F(\rho \theta)\, \rho^{n-1} \, d\rho = r^{n-1} F(r\theta).
		\] Substituting this into the expression yields
		\[
		\int_{\mathbb{R}^n} \delta(\|x\| - r)\, F(x)\, dx
		= r^{n-1} \int_{S^{n-1}(1)} F(r\theta)\, d\sigma(\theta),
		\]
		as required.
	\end{proof}
	
	\begin{lemma} \label{l3.2}	\cite[Ch.2]{dlkpro} If \(g\) is a smooth real function with simple zeros \(x_i,\) i.e.  \(g(x_i) = 0\) and \(g'(x_i) \neq 0\), then
		\[
		\delta(g(x)) = \sum_i \frac{\delta(x - x_i)}{|g'(x_i)|}.
		\]
	\end{lemma}

	\subsection{Tropical Nevanlinna theory in one variable}
	Now let's recall the Nevanlinna theory of tropical meromorphic functions in one dimension (see \cite{tropnevan}, \cite{cmft2012},  \cite{book}).\par
	
	For any real number $x\in\mathbb{T},$ we define $x^+ := \max\{x, 0\} .$ For a tropical meromorphic function $f$  we set $ f^{+}(x) := \max(f(x), 0).$ Define the proximity function of $f$ by taking the average of the positive parts of $f$  at the endpoints in $[-r, +r]$ as follows:	
	\[
	m(r, f) := \frac{f^+(r) + f^+(-r)}{2}.
	\]
	Denote by  \( n(r, f) \)  the number of poles of $f$ in the interval \((-r, r)\), counted with multiplicity. The counting function of $f$ is defined by
	\[
	N(r, f) := \frac{1}{2} \int_0^r n(t, f) \, dt = \frac{1}{2} \sum_{\nu=1}^{N} (r - |b_\nu|),
	\] where $b_\nu$ are all poles of $f$ (counting multiplicities) in $(-r, r).$	The characteristic function of $f$ is defined by:
	\[
	T(r, f) := m(r, f) + N(r, f).
	\]
	
	The tropical Poisson-Jensen formula is given as follows.\par
	
	\begin{lemma}[Poisson-Jensen formula]\cite{tropnevan}\label{jensen lemma}
		Suppose that \( f \) is a tropical meromorphic function on \( [-r, r] \), for some \( r > 0 \) and denote the roots of \( f \) in this interval by \( a_\mu, \, \mu = 1, \ldots, M \), and the poles by \( b_\nu,\,  \nu = 1, \ldots, N \), where roots and poles are listed according to their multiplicities. Then for any \( x \in (-r, r) \) we have
		\begin{align*}
			f(x) &= \frac{1}{2} \left\{ f(r) + f(-r) \right\} + \frac{x}{2r} \left\{ f(r) - f(-r) \right\} \nonumber \\
			&\quad - \frac{1}{2r} \sum_{\mu=1}^{M} \left\{ r^2 - |a_\mu - x| r - a_\mu x \right\}
			+ \frac{1}{2r} \sum_{\nu=1}^{N} \left\{ r^2 - |b_\nu - x| r - b_\nu x \right\}. \label{p-jensen formula}
		\end{align*}
		
	\end{lemma}
	
	In particular,
	\begin{equation*}
		f(0) = \frac{f(r) + f(-r)}{2} - \frac{1}{2} \sum_{\mu=1}^{M} (r - |a_\mu|) + \frac{1}{2} \sum_{\nu=1}^{N} (r - |b_\nu|). \label{jensen formula}
	\end{equation*}
	which  satisfies
	\[
	T(r, f) =T(r, 1_{\mathbb{T}}\oslash f) +f(0).
	\]
	This relation serves as a weak analogue version of Nevanlinna's first main theorem. Furthermore,  a general form of the first main theorem is shown as follows.\par
	
	\begin{lemma}[First main theorem]\cite{tropnevan} Let \( f \) be a tropical meromorphic function of one variable. Then for any $a\in\mathbb{T}$ satisfying $a < L_f := \inf \{f(b): b \text{ is a pole of } f\}, $ we have
		\[
		T(r, \mathbf{1}_{\mathbb{T}} \oslash (f \oplus a)) = T(r, f) + O(1).
		\]
	\end{lemma}

The estimation on tropical logarithmic derivative in one variable is given by Laine and Tohge \cite{laine-tohge}.\par

\begin{lemma}\label{one varialbe ldl}\cite{laine-tohge} (see also \cite[Theorem 3.24]{book}) Let $f$ be tropical meromorphic. Then, for all $\alpha>1$ and $r>0,$ we have 
$$m(r, f(x+c)\oslash f(x))\leq \frac{16|c|}{r+|c|}\frac{1}{\alpha-1}T(\alpha(r+|c|), f)+\frac{|f(0)|}{2}.$$

\end{lemma}

	\subsection{Notations for tropical Nevanlinna theory in higher dimension}
	Now we establish some basic notations for the higher-dimensional tropical value distribution theory.\par
	
	Firstly, we introduce the high-dimensional tropical proximity function \(m(r,f)\) which is naturally  defined as an average value over the unit sphere, can be expressed in Cartesian coordinates by using
	the Dirac delta function to restrict the integration to the sphere of radius $r$ as follows.\par
	
	\begin{definition}[Proximity function]\label{proxi} For a tropical meromorphic function $f$ on $\mathbb{R}^n,$ its proximity function \(m(r,f)\)  is defined by
		\begin{eqnarray}\label{prox}
			m(r,f) :=\frac{1}{\omega_n} \int_{S^{n-1}(1)}f^{+}(r\theta) d\sigma(\theta).
	\end{eqnarray} \end{definition}

	\begin{lemma}\label{L-3.7}
		It follows from the definition \eqref{prox} that
		\begin{align}\label{m-1}
			m(r,f) &=\frac{1}{\omega_n r^{n-1}} \int_{\mathbb{R}^n} f^{+}(x)\, \delta(\|x\| - r) \, dx,\\
			\label{m-2}&=\frac{1}{\omega_{n}}\int_{S^{n-1}(1)} m(r,f_{\theta})  d\sigma(\theta)\\
			\label{m-3}&=\frac{1}{\omega_n r^{n-1}} \int_{\mathbb{R}^{n}}m(r, f(x))\delta(\|x\|-r)dx,\end{align}  where $m(r, f_{\theta}):=\frac{f_\theta{}^+(r) + f_{\theta}^+(-r)}{2}.$
	\end{lemma}	
	
	\begin{proof} It follows directly from \eqref{E2} to get that $$m(r, f)=\frac{1}{\omega_n} \int_{S^{n-1}(1)}f^{+}(r\theta) d\sigma(\theta)= \frac{1}{\omega_n r^{n-1}} \int_{\mathbb{R}^n} f^{+}(x)\, \delta(\|x\| - r) \, dx.$$   For any direction $\theta\in S^{n-1}(1),$ we consider also its inverse direction $-\theta,$ and deduce that
		\begin{eqnarray*}m(r,f) &=&\frac{1}{\omega_n} \int_{S^{n-1}(1)}f^{+}(r\theta) d\sigma(\theta) \\
			&=&\frac{1}{\omega_n} \int_{S^{n-1}(1)}\frac{f^{+}(r\theta)+f^{+}(-r\theta)}{2} d\sigma(\theta)\\
			&=&\frac{1}{\omega_n} \int_{S^{n-1}(1)}m(r, f_{\theta}(t))d\sigma(\theta)\\
		\end{eqnarray*}
		
		Using \eqref{E2} again to get that
		$$
		\frac{1}{\omega_n} \int_{S^{n-1}(1)}m(r, f_{\theta}(t)) d\sigma(\theta)
		=\frac{1}{\omega_n r^{n-1}} \int_{\mathbb{R}^{n}}m(r, f(x))\delta(\|x\|-r)dx.$$
	\end{proof}
	
	Similarly, we define the high-dimensional tropical counting function as follows.\par
	
	\begin{definition}[Counting function] \label{ctf}Let $f$ be a tropical meromorphic function in $\mathbb{R}^n,$
		Denote by
		$n(t,f_{\theta})=\sum_{\|x\|<t, J_{f_{\theta}}(x; 1)<0} |J_{f_{\theta}}(x; 1)|$ the number of poles of $f_{\theta},$ counted with multiplicity. The tropical counting function for the poles in the ball $B[r]$  is defined by
		\begin{align}\label{dfn1}
			N(r,f):=\frac{1}{2}\int_{0}^{r}n(t,f)dt,
		\end{align}
		where \begin{align*}
			n(t,f)&=\frac{1}{\omega_n}\int_{S^{n-1}(1)} n(t,f_{\theta})\,d\sigma(\theta).
		\end{align*}
		Note that one can also denote $n(t,f_{\theta}):=\sum_{\|x\|<t, J_{f_{\theta}}(x; -1)<0} |J_{f_{\theta}}(x; -1)|$ due to \eqref{J=}.
	\end{definition}
	
	\begin{example}\label{nu}
		Let $f: \mathbb{R}^n \to \mathbb{T}$ be a tropical rational function in \eqref{tro rat}, and $B_t:=\{x\in \mathbb{R}^n:\|x\|<t\}$.
		Then non-smooth locus $\Sigma_f$ is a locally finite $(n-1)$-dimensional polyhedral complex and
		\[
		n(t,f)=\int_{\Sigma_f\cap B_t}\nu_f(x)\,d\mathcal{H}^{\,n-1}(x),
		\] where $d\mathcal{H}^{\,n-1}(x)$ is the Hausdorff measure on $\mathbb{R}^{n-1}.$  
	\end{example}
	
	\begin{proof} Set $x\in \mathbb{R}^n,$ and $\langle \mathbf{m}_i,x\rangle	= m_{i1} x_1 + m_{i2} x_2 + \cdots + m_{in} x_n. $  Write $L_i(x):= a_i \otimes x^{\otimes \mathbf{m}_i}=a_i+\langle \mathbf{m}_i,x\rangle$ and $M_j(x):= b_j \otimes x^{\otimes \mathbf{n}_j} =b_j+\langle \mathbf{n}_j, x\rangle.$ Now we can write $f=\frac{P}{Q}\oslash,$ where $P(x):= \bigoplus_{i=0}^{p} (a_i+\langle \mathbf{m}_i,x\rangle) $ and 	$Q(x):= \bigoplus_{j=0}^{q} (b_j+\langle \mathbf{n}_j, x\rangle).$	For $i\neq i'$, the equality $L_i=L_{i'}$ defines the hyperplane
		\[
		H^P_{ii'}:=\{x:\langle m_i-m_{i'},x\rangle=a_{i'}-a_i\},
		\]
		and since $P$ has finitely many terms, only finitely many such hyperplanes occur. The non-smooth locus $\Sigma_P$ consists of those points where at least two $L_i$ tie for the maximum, hence it is contained in the finite union $\bigcup_{i<i'}H^P_{ii'}$ . Thus $\Sigma_P$ is a locally finite $(n-1)$-dimensional polyhedral complex. The same argument applies to $Q$, hence $\Sigma_Q$ is locally finite. If $f$ is non-differentiable at $x$, then at least one of $P$ or $Q$ is non-differentiable at $x$, which implies
		\[
		\Sigma_f \subset \Sigma_P\cup \Sigma_Q,
		\]
		and therefore $\Sigma_f$ is also a locally finite $(n-1)$-dimensional polyhedral complex. In particular, $\Sigma_f\cap B_t$ has finite $\mathcal{H}^{\,n-1}$-measure for every $t>0$.

		For a fixed $\varphi\in S^{n-1}(1)$, the restriction $f_\varphi(r):=f(r\varphi)$ is piecewise linear in one variable with only finitely many breakpoints in $(0,t]$ by the local finiteness of $\Sigma_f$. By the one-dimensional definition of $n(t,f_\varphi)$ (Definition~\ref{ctf}),
		\[
		n(t,f_\varphi)=\sum_{\substack{x\in \Sigma_f\cap [-t,t]\varphi\\ J_f(x;\varphi)<0}} |J_f(x;\varphi)|.
		\]
		Averaging over directions and using the above expression gives
		\[
		n(t,f) = \frac{1}{\omega_n}\int_{S^{n-1}(1)}
	\sum_{\substack{x\in \Sigma_f\cap [-t,t]\varphi\\ J_f(x;\varphi)<0}} |J_f(x;\varphi)|d\sigma(\varphi).
		\]
		Since $\sigma(S^{n-1}(1))<\infty$, $\mathcal{H}^{\,n-1}(\Sigma_f\cap [-t,t]\varphi)<\infty$, and $|J_f|$ is bounded on the product space $S^{n-1}(1)\times(\Sigma_f\cap [-t,t]\varphi)$, the integrand is absolutely integrable with respect to the product measure $\sigma\times\mathcal{H}^{\,n-1}$:
		\begin{align*}
			&\frac{1}{\omega_n}\int_{S^{n-1}(1)}
			\sum_{\substack{x\in \Sigma_f\cap [-t,t]\varphi\\ J_f(x;\varphi)<0}} |J_f(x;\varphi)|d\sigma(\varphi)\\
			=&\int_{\Sigma_f\cap B_t}
			\frac{1}{\omega_n}\int_{\{\varphi\in S^{n-1}(1):\,J_f(x;\varphi)<0\}}
			|J_f(x;\varphi)|\,d\sigma(\varphi)\,d\mathcal{H}^{\,n-1}(x)\\
			=& \int_{\Sigma_f\cap B_t} \nu_f(x)\,d\mathcal{H}^{\,n-1}(x),
		\end{align*}
	\end{proof}

	\begin{lemma}\label{L-3.9}
		It follows from the definition  \eqref{dfn1} that
		\begin{align*}\label{dfn2}
			N(r,f) &=\frac{1}{\omega_n r^{n-1}}\int_{\mathbb{R}^n}N(r, f(x))\,\delta(\|x\| - r)\,\,dx,\\
			&=\frac{1}{\omega_{n}}\int_{S^{n-1}(1)} N(r,f_{\theta})  d\sigma(\theta).\end{align*}
	\end{lemma}

	\begin{proof}
		By Definition \eqref{dfn1}, we rewrite the integral as
		\begin{align*} N(r,f(x))=\frac{1}{2}\int_{0}^{r}\frac{1}{\omega_{n}}\left(\int_{S^{n-1}(1)}\sum_{\|x\|<t, J_{f_{\theta}}(x; 1)<0} |J_{f_{\theta}}(x; 1)|d\sigma(\theta)\right)dt.\end{align*} For any fixed $\theta,$ we consider the one variable function $f_\theta$  and have
		\begin{align*} N(r,f_{\theta})=\frac{1}{2}\int_{0}^{r}n(t,f_{\theta})\ dt=\frac{1}{2}\sum_{\|x\|<r, J_{f_{\theta}}(x; 1)<0} |J_{f_{\theta}}(x; 1)|(r-\|x\|).  \end{align*}
		Then combing the above two equalities, we obtain
		\begin{align*}N(r,f)&=\frac{1}{2\omega_{n}}\int_{S^{n-1}(1)}\int_{0}^{r} \sum_{\|x\|<t, J_{f_{\theta}}(x; 1)<0} |J_{f_{\theta}}(x; 1)|\ dt d\sigma(\theta)\\
			&= \frac{1}{2\omega_{n}}\int_{S^{n-1}(1)}\sum_{\|x\|<r, J_{f_{\theta}}(x; 1)<0} |J_{f_{\theta}}(x; 1)|(r-\|x\|) d\sigma(\theta)\\
			&= \frac{1}{\omega_{n}}\int_{S^{n-1}(1)} N(r,f_{\theta} )  d\sigma(\theta).\end{align*}
		Together with \eqref{E2}, we get that
		\begin{align*}N(r,f)=\frac{1}{\omega_{n}}\int_{S^{n-1}(1)} N(r,f_{\theta} )d\sigma(\theta)=\frac{1}{\omega_{n}r^{n-1}}\int_{\mathbb{R}^{n}} N(r,f(x))\delta(\|x\|-r) dx.\end{align*}
	\end{proof}

	As usually,	we call also
	\[
	T(r,f) := m(r,f) + N(r,f)
	\]
	to be the characteristic function of  a tropical meromorphic function $f.$ By the above two lemmas, we have the following property for $T(r,f).$\par
	
	\begin{lemma}	\label{L-10}
		\begin{align*} T(r,f)=\frac{1}{\omega_{n}}\int_{S^{n-1}(1)}[m(r,f_{\theta})+N(r,f_{\theta})]d\sigma(\theta)=:\frac{1}{\omega_{n}}\int_{S^{n-1}(1)}T(r,f_{\theta})d\sigma(\theta).\end{align*}\end{lemma}
	
	\begin{lemma}	$T(r, f)$ is a convex function about $r$.	\end{lemma}
	\begin{proof}By the above lemma, we have
		\[
		T(r, f) = \frac{1}{\omega_n} \int_{S^{n-1}(1)} T(r, f_\theta)  d\sigma(\theta),
		\]
		where \( f_\theta(t) = f(t\theta) \) and \( \omega_n \) is the surface area of \( S^{n-1}(1) \).
		
		For each fixed \( \theta \), it is known that \( T(r, f_\theta) \) is convex in \( r \) (see  \cite[Corollay 3.10]{book}). Thus for any \( r_1, r_2 > 0 \) and \( \lambda \in [0,1] \):
		\[
		T(\lambda r_1 + (1-\lambda)r_2, f_\theta) \leq \lambda T(r_1, f_\theta) + (1-\lambda) T(r_2, f_\theta).
		\]
		By integrating the above inequalities over \( S^{n-1}(1) \), we obtain
		\begin{eqnarray*}
		&&\frac{1}{\omega_n} \int_{S^{n-1}(1)} T(\lambda r_1 + (1-\lambda)r_2, f_\theta)  d\sigma(\theta) \\&\leq& \frac{1}{\omega_n} \int_{S^{n-1}(1)} \left[ \lambda T(r_1, f_\theta) + (1-\lambda) T(r_2, f_\theta) \right] d\sigma(\theta).
		\end{eqnarray*}
		This implies that
		\[
		T(\lambda r_1 + (1-\lambda)r_2, f) \leq \lambda T(r_1, f) + (1-\lambda) T(r_2, f).
		\]
		and thus \( T(r, f) \) is convex in \( r \).
		
	\end{proof}
	
	\begin{lemma}\label{pro}
		Given a positive number $\alpha$ and  $f : \mathbb{R}^n \to \mathbb{R} \cup \{+\infty\}$ then
		\begin{align*}
			&m(r, f^{\otimes \alpha}) = m(r, \alpha f) = \alpha m(r, f), \nonumber \\
			&N(r, f^{\otimes \alpha}) = N(r, \alpha f) = \alpha N(r, f), \nonumber \\
			&T(r, f^{\otimes \alpha}) = T(r, \alpha f) = \alpha T(r, f),\nonumber \\
			&m(r, f \oplus g) \leq m(r, f) + m(r, g),\nonumber \\
			&m(r, f \otimes g) \leq m(r, f) + m(r, g),\nonumber  \\
			&N(r, f \otimes g) \leq N(r, f) + N(r, g), \nonumber \\
			&T(r, f \otimes g) \leq T(r, f) + T(r, g),\\
			&T(r, f \oplus g)  \leq T(r, f) + T(r, g).
		\end{align*}
	\end{lemma}

	\begin{proof}One can directly deduce them according to the definitions of $m(r, f),$ $N(r, f)$ and $T(r,f).$ It is known that the one dimensional case has been proved by Laine and Tohge (see \cite{book}). Thus for any fixed $\theta,$ we have all the properties for the function $f_{\theta}(t)$ which is regarded as a tropical meromorphic function in one variable. Then it is easy to get the lemma by the above lemmas.  We omit the details.\end{proof}
	
	\subsection{The first main theorem in higher dimension}
	
	We get the softed form of the tropical first main theorem in higher dimension.\par
	
	\begin{lemma}\label{fmtn}
		Let $f \colon \mathbb{R}^{n} \rightarrow \mathbb{R} \cup \{+\infty\}$ be a tropical meromorphic function. Then the following identity holds:
		\begin{equation*}
			T(r,f) - T(r,-f) = f(0).
		\end{equation*}
	\end{lemma}
	
	\begin{proof}
		For each fixed direction $\theta \in S^{n-1}(1)$, consider the one-dimensional slice $f_{\theta}(t) := f(t\theta)$, which is a tropical meromorphic function on $\mathbb{R}$.  By the one-dimensional tropical Jensen formula \eqref{jensen formula}, we have
		\begin{equation*}
			m(r, f_{\theta})+N(r, f_{\theta}) - m(r, -f_{\theta})-N(r, -f_{\theta}) = f_{\theta}(0) = f(0).
		\end{equation*}
		Averaging both sides of the identity over all directions $\theta \in S^{n-1}(1)$, we obtain
		\begin{align*}
			&\frac{1}{\omega_{n}} \int_{S^{n-1}(1)} \left[	m(r, f_{\theta})+N(r, f_{\theta}) - m(r, -f_{\theta})-N(r, -f_{\theta}) \right] d\sigma(\theta)\\
			=& m(r,f)+N(r,f)-m(r,-f)+N(r,-f)\\
			= &\frac{1}{\omega_{n}} \int_{S^{n-1}(1)} f(0)  d\sigma(\theta) \\= &f(0).
		\end{align*}
		On the other hand, by Lemma \ref{L-10} we have
		\begin{align*}
			\frac{1}{\omega_{n}} \int_{S^{n-1}(1)}	m(r, f_{\theta})+N(r, f_{\theta}) d\sigma(\theta) &= T(r, f)\end{align*}
		and \begin{align*}	
			\frac{1}{\omega_{n}} \int_{S^{n-1}(1)} 	m(r, -f_{\theta})+N(r, -f_{\theta}) d\sigma(\theta) &= T(r, -f).
		\end{align*}
		Therefore,
		\begin{equation*}
			T(r, f) - T(r, -f) = f(0).
		\end{equation*}
	\end{proof}
	From the proof of Lemma \ref{fmtn}, we derive the high-dimensional tropical Jensen formula as follows:	
	\begin{align}\label{Jensen-1}
		\frac{1}{\omega_{n}} \int_{S^{n-1}(1)}f(r\theta)d\sigma(\theta)=N(r,-f)-N(r,f)+O(1)
	\end{align}
	which is frequently used in the subsequent text.
	We can also establish the  general form of the tropical first main theorem in higher dimension.\par
	
	\begin{theorem}	[First main theorem in higher dimension] \label{T3.15}
		For any $a\in\mathbb{T}$ satisfying $a < L_f := \inf \{f(b): b \text{ is a pole of } f\}, $ we have
		\[
		T(r, \mathbf{1}_{\mathbb{T}} \oslash (f \oplus a)) = T(r, f) + O(1).
		\]
	\end{theorem}

	\begin{proof}Making use of Lemma \ref{fmtn}, we immediately conclude that
		\[
		\begin{aligned}
			T(r, 1_{\mathbb{T}} \oslash (f \oplus a))
			&= T(r, f \oplus a) - (f(0,\cdots,0)\oplus a) \\
			&\leq T(r, f) + T(r, a) - (f(0,\cdots,0)\oplus a)\\
			&\leq T(r, f) + \max\{a, 0\} - (f(0,\cdots,0)\oplus a)
		\end{aligned}
		\]
		for any \(a \in \mathbb{R}\) and for any \(r > 0\). Here we also used the inequality $T(r, f \oplus g)  \leq T(r, f) + T(r, g).$ It follows from Lemma \ref{pro} that \(T(r, a) = a^+ = \max\{a, 0\}\).  In addition,
		\begin{eqnarray*}
			&&\max\{a, 0\} - \max\{f(0,\cdots,0), a\} \\&=&
			\begin{cases}
				a - f(0,\cdots,0) \leq 0; & a > 0, f(0,\cdots,0) \geq a \\
				a - a = 0; & a > 0, f(0,\cdots,0) < a \\
				0 - f(0,\cdots,0) \leq |a|; & a \leq 0, f(0,\cdots,0) \geq a \\
				0 - a = |a|; & a \leq 0, f(0,\cdots,0) < a
			\end{cases}
			\\&\leq& |a|.
		\end{eqnarray*}
		
		To obtain the asserted asymptotic equality, suppose first that \(f\) has at least one pole and that \(-\infty < a < L_f\). In this case, we have \(N(r, \max\{f, a\}) = N(r, f)\). Therefore,
		\[
		\begin{aligned}
			T(r, 1_{\circ} \oslash (f \oplus a)) &= T(r, f \oplus a) - (f(0,\cdots,0)\oplus a)  \\
			&= m(r, \max\{f, a\}) + N(r, \max\{f, a\}) - \max\{f(0,\cdots,0), a\} \\
			&\geq m(r, f) + N(r, f) - \max\{f(0,\cdots,0), a\} \\
			&\geq T(r, f) - \max\{f(0), a\},
		\end{aligned}
		\]
		according to the monotonicity of \(m(r, f)\) with respect to \(f\).
		
		Finally, if \(L_f = +\infty\), that is, if \(f\) has no poles, the asymptotic equality holds as well. In fact, because of \(T(r, f) = m(r, f)\) then and \(f \oplus a \geq f\) for any \(a \in \mathbb{R}\), we have
		\begin{align*}
			T(r, 1_{\circ} \oslash (f \oplus a)) &= T(r, f \oplus a) - (f(0,\cdots,0)\oplus a) \\
			&= T(r, \max\{f, a\}) - \max\{f(0,\cdots,0)), a\} \\
			&\geq m(r, \max\{f, a\}) - \max\{f(0,\cdots,0), a\} \\
			&\geq m(r, f) - \max\{f(0,\cdots,0)), a\} \\
			&= T(r, f) - \max\{f(0,\cdots,0), a\}.
		\end{align*}
	\end{proof}
	
	\section{The tropical version of the logarithmic derivative lemma}\label{Sec-4}
	
	For a high-dimensional tropical meromorphic function \( f: \mathbb{R}^{n} \to \mathbb{R} \cup \{ +\infty \} \), the order is defined as follows:
	\begin{align*}
		\rho(f) := \limsup_{r \to \infty} \frac{\log T(r, f)}{\log r},
	\end{align*}
	where \( T(r, f) \) is the characteristic function that quantifies the growth rate of the function. The hyper-order \( \rho_2(f) \) is defined as follows:
	\begin{align*}
		\rho_2(f) := \limsup_{r \to \infty} \frac{\log \log T(r, f)}{\log r}.
	\end{align*}
Subnormal growth (also called minimal hypertype) for $f$  is defined as follows:
	\begin{align*}
		\limsup_{r \to \infty} \frac{\log T(r, f)}{r}.
	\end{align*}

	We now present a key lemma that will be useful in proving the tropical version of the logarithmic derivative lemma in higher dimension.\par
	
	\begin{lemma}\label{stolllemma}
		Let $H(x)$ be a real integrable  function on $\mathbb{R}^{n}.$ Then we have 		
		\begin{align*}
			\int_{\mathbb{R}^n} H(x) \delta(\|x\| - r)dx= \int_{\|w\| < r} \left( \frac{r}{p_r(w)} \cdot  \int_{\mathbb{R}}  H((w,\zeta)) \delta\big(|\zeta| - p_r(w)\big)d\zeta \right) dw,
		\end{align*}where $p_{r}(w):=\sqrt{r^2-\|w\|^2}>0.$
	\end{lemma}
	
	\begin{proof}Set $x = (w,\zeta) \in \mathbb{R}^{n-1} \times \mathbb{R},$  and $\|x\|^2=\|w\|^2 + |\zeta|^2.$ Define a function 
		$$g(\zeta):=\|x\| - r=\sqrt{\|w\|^2 + |\zeta|^2} -r, \,\, (\|w\|<r),$$ whose zeros $\zeta_i:=\pm\sqrt{r^2 - \|w\|^2}\neq 0.$ It follows that $g'(\zeta)=\frac{\zeta}{\sqrt{\|w\|^2+|\zeta|^2}}$ and $g'(\zeta_i)\neq 0.$
		Then by Lemma \ref{l3.2}  we have
		\begin{align*}\delta(\|x\| - r)&=\delta(g(\zeta))\\&=
			\sum_{i} \frac{\delta(\zeta - \zeta_i)}{|g'(\zeta_i)|}\\&= \frac{r}{\sqrt{r^2 - \|w\|^2}} \left(\delta\left(\zeta - \sqrt{r^2 - \|w\|^2}\right) + \delta\left(\zeta + \sqrt{r^2 - \|w\|^2}\right) \right),
		\end{align*} where $\|w\|<r.$
		
		Set $p_{r}(w):=\sqrt{r^2-\|w\|^2}.$ Then integrating over $\mathbb{R}^n$ we obtain		
		\begin{eqnarray*}
			&&\int_{\mathbb{R}^n}  H(x)\delta(\|x\| - r)dx  \nonumber\\
			&=&\int_{\|w\| < r} \frac{r H(w, \zeta) }{\sqrt{r^2 - \|w\|^2}} \left(\delta\left(\zeta - \sqrt{r^2 - \|w\|^2}\right) + \delta\left(\zeta + \sqrt{r^2 - \|w\|^2}\right) \right)dw\\
			&=& \int_{\|w\| < r} \left( \frac{r}{p_r(w)} \cdot  \int_{\mathbb{R}}  H((w,\zeta)) \delta\big(|\zeta| - p_r(w)\big) d\zeta \right) dw
		\end{eqnarray*}  where $p_{r}(w)>0.$
		
	\end{proof}
	
	So in terms of Lemma \ref{L-3.7} and Lemma \ref{stolllemma},  we have
	
	\begin{align}\label{E10}
		m(r,f) &= \frac{1}{\omega_n r^{n-1}} \int_{\mathbb{R}^n}m(r,f(x)) \delta(\|x\| - r) \, dx \nonumber\\
		&= \frac{1}{\omega_n r^{n-1}} \int_{\|w\| < r} \left( \frac{r}{p_r(w)} \cdot  \int_{\mathbb{R}}   m(r, f(w,\zeta)) \left(\delta\big(|\zeta| - p_r(w) \right) d\zeta \right) dw \nonumber\\
		&= \frac{1}{\omega_n r^{n-1}} \int_{\|w\| < r} \left( \frac{2r}{p_r(w)} \cdot   m(p_r(w), f_w(\zeta))\right) dw.
	\end{align}
	Similarly, for $n(r,f)$ and $N(r,f)$, it follows from Lemma \ref{L-3.9} and Lemma \ref{stolllemma} that
	\begin{align}
		n(r,f) &= \frac{1}{\omega_n r^{n-1}} \int_{\mathbb{R}^n} n(r, f(x)) \delta(\|x\| - r) \, dx \nonumber\\
		&= \frac{1}{\omega_n r^{n-1}} \int_{\|w\| < r} \left( \frac{r}{p_r(w)} \cdot  \int_{\mathbb{R}}  n(r, f(w,\zeta)) \left( \delta\big(\|\zeta\| - p_r(w) \right) d\zeta \right) dw	\nonumber\\
		&= \frac{1}{\omega_n r^{n-1}} \int_{\|w\| < r} \left( \frac{2r}{p_r(w)} \cdot   n(p_r(w), f_w(\zeta))\right) dw
	\end{align} and
	\begin{align}
		N(r,f) &= \frac{1}{\omega_n r^{n-1}} \int_{\mathbb{R}^n} N (r, f(x)) \delta(\|x\| - r) \, dx \nonumber\\
		&= \frac{1}{\omega_n r^{n-1}} \int_{\|w\| < r} \left( \frac{r}{p_r(w)} \cdot  \int_{\mathbb{R}}  N(r, f(w,\zeta)) \left( \delta\big(\|\zeta\| - p_r(w) \right) d\zeta \right) dw	\nonumber\\
		&= \frac{1}{\omega_n r^{n-1}} \int_{\|w\| < r} \left( \frac{2r}{p_r(w)} \cdot   N(p_r(w), f_w(\zeta))\right) dw.
	\end{align}
	So, we have
	\begin{align}\label{Tdlk}
		T(r,f) = \frac{1}{\omega_n r^{n-1}} \int_{\|w\| < r} \left( \frac{2r}{p_r(w)} \cdot   T(p_r(w), f_w(\zeta))\right) dw,
	\end{align}
	
	The following Borel lemma was obtained by Cao and Zheng.\par
	
	\begin{lemma}\label{newloga}\cite{hypersurfaces}
		Let \( T(r) \) be a nondecreasing positive, convex, continuous function on \([1, +\infty)\) with
		\[
		\liminf_{r \to \infty} \frac{\log T(r)}{r} = 0.
		\]
		Then for the function
		\[
		\phi(r) := \max_{1 \leq t \leq r} \left\{ \left( \frac{t}{\log T(t)} \right)^\delta \right\}, \quad \delta \in \left( 0, \frac{1}{2} \right),
		\]
		we have
		\[
		T(r) \leq T(r + \phi(r)) \leq (1 + \varepsilon(r))T(r),
		\]
		where \( \varepsilon(r) \to 0 \) as \( r \) tends to infinity outside of a set of zero lower density measure \( E \), i.e.,
		\[
		\underline{\text{dens}} E = \liminf_{r \to \infty} \frac{1}{r} \int_{E \cap [1, r]} dt = 0.
		\]
		Especially, for any fixed positive real value \( c(\neq 0) \),
		\[
		T(r) \leq T(r + c) \leq (1 + \varepsilon(r))T(r), \; r \notin E \to \infty.
		\]		Furthermore, if the growth assumption is changed into
		\begin{equation*}\label{equ1}\limsup_{r\rightarrow\infty} \frac{\log T(r)}{r}=0,\end{equation*} then
		the exceptional set $E$ is a set with zero upper density measure,
		i.e.,
		$$\overline{dens}E=\limsup_{r\rightarrow\infty}\frac{1}{r}\int_{E\cap[1,r]}dt=0.$$
	\end{lemma}
	
	Now we extend the tropical version of the logarithmic derivative lemma for difference operator $x+c$ under the subnormal growth (or called minimal hypertype) \cite{hypersurfaces} from one variable to several variables.\par
	
	\begin{theorem}\label{loga}
		Let \( c \in \mathbb{R}^n \setminus \{0\} \). If \( f \) is a tropical meromorphic function \(f:\mathbb{R}^n \rightarrow \mathbb{T} \) with
		
		\begin{align}\label{loga*}
			\limsup_{r \to \infty} \frac{\log T(r, f)}{r} = 0,
		\end{align}
		then
		\[
		m(r, f(x + c) \oslash f(x)) = o(T(r, f)),
		\]
		where \( r \) runs to infinity outside of a set of zero upper density measure \( E \), i.e.,
		\[
		\overline{\text{dens}} E = \limsup_{r \to \infty} \frac{1}{r} \int_{E \cap [1, r]} dt = 0.
		\]
	\end{theorem}

	\begin{proof}	Set $x = (w,\zeta) \in \mathbb{R}^{n-1} \times \mathbb{R},$  and $\|x\|^2=\|w\|^2 + |\zeta|^2.$  Using Lemma \ref{stolllemma} (or \eqref{E10}), we have
		\begin{align*}
			m(r,f)= \frac{1}{\omega_n r^{n-1}} \int_{\|w\| < r} \left( \frac{2r}{p_r(w)} \cdot   m(p_r(w), f_w(\zeta))\right) dw,
		\end{align*}where $p_r(w) := \sqrt{r^2 - \|w\|^2}>0.$
		Then for the difference \( f(x + c) - f({x}) \), with \( c_\zeta=(0, \ldots, \tilde{c_\zeta},\ldots,0) \in \mathbb{R}^n \), we have
		\begin{align*}
			m(r, f(x + c_\zeta) - f({x}))	&=\frac{1}{\omega_n r^{n-1}} \int_{\|w\| < r}
			\frac{2r}{p_r(w)}m\left( p_r(w), \, f_{[w]}(\zeta +\tilde{c_\zeta}) - f_{[w]}(\zeta) \right) dw,
		\end{align*} where $ c_\zeta$ is the element of $c$ in terms of $\zeta.$\par
		
		By the estimation on tropical logarithmic derivative in one variable (Lemma \ref{one varialbe ldl},
		for all $\alpha(>1)$ and all $p_{r}(w)>0,$ we have
		\begin{align*}
			&\frac{2r}{p_r(w)}m\left( p_r(w), \, f_{[w]}(\zeta +\tilde{c_\zeta}) - f_{[w]}(\zeta) \right)\\\leq&\frac{32|\tilde{c_\zeta}|}{p_r(w)+|\tilde{c_\zeta}|}\frac{1}{\alpha-1}	\frac{r}{p_r(w)} T\left(\alpha( p_r(w)+|\tilde{c_\zeta}|),f_{[w]}\right)
			+\frac{r|f_{[w]}(0)|}{p_r(w)}.
		\end{align*} 	
		Hence we deduce that 
		\begin{eqnarray}\label{AAA}
			&&m(r, f(x + c_\zeta) - f({x}))\nonumber	\\
			&\leq&\frac{1}{\omega_n r^{n-1}} \int_{\|w\| < r}\frac{32|\tilde{c_\zeta}|}{p_r(w)+|\tilde{c_\zeta}|}\frac{1}{\alpha-1}	\frac{r}{p_r(w)} T\left(\alpha( p_r(w)+|\tilde{c_\zeta}|),f_{[w]}\right) dw\nonumber	\\
			&&+\frac{1}{\omega_n r^{n-1}} \int_{\|w\| < r}\frac{r M}{p_r(w)} dw,
		\end{eqnarray}
		where $M$ is a positive bounded value.\par

		Let $\mathcal A_\varepsilon:=\left\{w\in\mathbb{R}^{n-1}:   \left\|w\right\|\le\left(1-\varepsilon\right)r\,\right\}=\{w\in\mathbb{R}^{n-1}:  \|w\|<r\},$ where $\varepsilon\in\left(0,1\right)$. Define
		\begin{align}
			I_{1}
			:= \frac{1}{\omega_n r^{\,n-1}}
			\int_{\mathcal A_\varepsilon}
			\frac{32\left|\tilde c_\zeta\right|}{p_r\left(w\right)+\left|\tilde c_\zeta\right|}\cdot\frac{1}{\alpha-1}\cdot\frac{r}{p_r\left(w\right)}\,
			T\left(\alpha\left(p_r\left(w\right)+\left|\tilde c_\zeta\right|\right),\,f_{[w]}\right)\,dw .
			\label{np1.np2}
		\end{align}
		On $\mathcal A_\varepsilon$, we have $p_r\left(w\right)\ge r\sqrt{2\varepsilon-\varepsilon^2}$, which implies
		\begin{align}
			\frac{r}{p_r\left(w\right)}\le \frac{1}{\sqrt{2\varepsilon-\varepsilon^2}},
			\qquad
			\frac{\left|\tilde c_\zeta\right|}{p_r\left(w\right)+\left|\tilde c_\zeta\right|}
			\le \frac{\left|\tilde c_\zeta\right|}{r\sqrt{2\varepsilon-\varepsilon^2}} .
			\label{np1.np3}
		\end{align}
		Substituting \eqref{np1.np3} into \eqref{np1.np2} yields
		\begin{align}
			I_{1}
			\le \frac{32}{\alpha-1}\cdot\frac{\left|\tilde c_\zeta\right|}{r}\cdot\frac{1}{2\varepsilon-\varepsilon^2}\cdot
			\frac{1}{\omega_n r^{\,n-1}}
			\int_{\mathcal A_\varepsilon}
			T\left(\alpha\left(p_r\left(w\right)+\left|\tilde c_\zeta\right|\right),\,f_{[w]}\right)\,dw .
			\label{np1.np4}
		\end{align}\par
		
		Define $p_{\alpha\left(r+\left|\tilde c_\zeta\right|\right)}\left(w\right):=\sqrt{\alpha^2\left(r+\left|\tilde c_\zeta\right|\right)^2-\left\|w\right\|^2}.$ Then we have
		\begin{eqnarray*}&&\left(p_{\alpha\left(r+\left|\tilde c_\zeta\right|\right)}\left(w\right)\right)^2- \left(\alpha\big(p_r(w)+|\tilde c_\zeta|\big)\right)^2
			\\&=&\big[\,\alpha^2(r+|\tilde c_\zeta|)^2 - \|w\|^2\,\big]
			- \alpha^2\left[\,p_r(w)^2 + 2|\tilde c_\zeta|\,p_r(w) + |\tilde c_\zeta|^2\,\right] \\
			&=& (\alpha^2 - 1)\,\|w\|^2 \;+\; 2\alpha^2|\tilde c_\zeta|\left(r - \sqrt{r^2 - \|w\|^2}\right)\\ 
			&\ge& 0,
		\end{eqnarray*}
		and thus 
		\begin{align}
			p_{\alpha\left(r+\left|\tilde c_\zeta\right|\right)}\left(w\right)
			\ge \alpha\left(p_r\left(w\right)+\left|\tilde c_\zeta\right|\right).
			\label{np1.np5}
		\end{align}
		This gives	\begin{align}
			T\left(\alpha\left(p_r\left(w\right)+\left|\tilde c_\zeta\right|\right),\,f_{[w]}\right)
			\le T\left(p_{\alpha\left(r+\left|\tilde c_\zeta\right|\right)}\left(w\right),\,f_{[w]}\right).
			\label{np1.np6}
		\end{align}
		Further, since $p_{\alpha\left(r+\left|\tilde c_\zeta\right|\right)}\left(w\right)\le \alpha\left(r+\left|\tilde c_\zeta\right|\right)$, we obtain
		\begin{align}
			T\left(p_{\alpha\left(r+\left|\tilde c_\zeta\right|\right)}\left(w\right),\,f_{[w]}\right)
			\le \frac{1}{2}\cdot\frac{2\alpha\left(r+\left|\tilde c_\zeta\right|\right)}{p_{\alpha\left(r+\left|\tilde c_\zeta\right|\right)}\left(w\right)}\,
			T\left(p_{\alpha\left(r+\left|\tilde c_\zeta\right|\right)}\left(w\right),\,f_{[w]}\right).
			\label{np1.np7}
		\end{align}\par
		
		Therefore, combining \eqref{np1.np4}-\eqref{np1.np7} with \eqref{Tdlk}, we have
		\begin{align}
			I_{1}
			\le \frac{16}{\alpha-1}\cdot\frac{\left|\tilde c_\zeta\right|}{r}\cdot\frac{1}{2\varepsilon-\varepsilon^2}\cdot
			\left(\frac{\alpha\left(r+\left|\tilde c_\zeta\right|\right)}{r}\right)^{n-1}\,
			T\left(\alpha\left(r+\left|\tilde c_\zeta\right|\right),\,f\right).
			\label{np1.np8}
		\end{align}
		
		Let $\delta\in\left(0,\tfrac12\right)$ and assume
		\begin{align}
			\alpha=\alpha\left(r\right)
			:= 1+\frac{\left(r+\left|\tilde c_\zeta\right|\right)^{\delta-1}}{\left(\log T\left(r+\left|\tilde c_\zeta\right|,f\right)\right)^{\delta}},
			\label{np1.np9}
		\end{align}
		Since \eqref{loga*},  by Lemma \ref{newloga}, there exist a set $E$ of upper density zero and a function $\eta\left(r\right)\to 0$ such that, for all $r\notin E$,
		\begin{align}
			T\left(\alpha\left(r+\left|\tilde c_\zeta\right|\right),\,f\right)
			&\le \left(1+\eta\left(r\right)\right)\,T\left(r+\left|\tilde c_\zeta\right|,f\right)\nonumber\\
			&\le \left(1+\eta\left(r\right)\right)\left(1+\eta_c\left(r\right)\right)\,T\left(r,f\right),
			\label{np1.np10}
		\end{align}
		If we choose
		\[
		\varepsilon:= 1 - \sqrt{1 - \sqrt{\frac{1}{(\alpha-1)(r+\left|\tilde c_\zeta\right|)}}}\to 0,
		\]
		then we have
		\begin{align*}
			\frac{1}{2\varepsilon-\varepsilon^2}=\sqrt{(\alpha-1)(r+\left|\tilde c_\zeta\right|)}.
		\end{align*}
		Additionally, we get
		\begin{align}
			\frac{1}{(\alpha-1)r}&=\frac{1}{(\alpha-1)(r+\left|\tilde c_\zeta\right|)}\frac{(r+\left|\tilde c_\zeta\right|)}{r} \nonumber\\
			&=\frac{\left(\log T\left(r+\left|\tilde c_\zeta\right|,f\right)\right)^{\delta}}{\left(r+\left|\tilde c_\zeta\right|\right)^{\delta}}\frac{(r+\left|\tilde c_\zeta\right|)}{r} .\label{np1.np13}
		\end{align}
		Note the assumption \eqref{loga*}, it follows that 
		\begin{align}
			\frac{1}{2\varepsilon-\varepsilon^2}	\frac{1}{(\alpha-1)r}=\frac{\left(\log T\left(r+\left|\tilde c_\zeta\right|,f\right)\right)^{\frac{\delta}{2}}}{\left(r+\left|\tilde c_\zeta\right|\right)^{\frac{\delta}{2}}}\frac{(r+\left|\tilde c_\zeta\right|)}{r} \to 0
		\end{align}
		and
		\begin{align}
			\left(\frac{\alpha\left(r+\left|\tilde c_\zeta\right|\right)}{r}\right)^{n-1}
			&=	\left( \alpha \right)^{n-1}  	\left(\frac{r+\left|\tilde c_\zeta\right|}{r}\right)^{n-1} \to 1+o\left(1\right).
			\label{np1.np11}
		\end{align}	\par

		Now, from \eqref{np1.np8}-\eqref{np1.np11}, we get 
		\begin{equation}\label{BBB} I_{1}=o\left(T\left(r,f\right)\right),\end{equation} where \( r \) runs to infinity outside of a set of zero upper density measure \( E \).

		Set
		\begin{align}
			I_{2}(r) := \frac{M}{\omega_n r^{n-1}} \int_{\|w\| < r} \frac{r}{p_r(w)} \, dw = \frac{M}{\omega_n r^{n-1}} \int_{\|w\| < r} \frac{r}{\sqrt{r^2 - \|w\|^2}} \, dw. \label{nppp1}
		\end{align}
		In polar coordinates, let \( w = t \xi \), with \( t \in [0, r] \), \( \xi \in S^{n-2} \), and thus \( dw = t^{n-2} \, dt \, d\sigma_{n-2}(\xi) \). This leads to the expression
		\begin{align}
			\int_{\|w\| < r} \frac{r}{\sqrt{r^2 - \|w\|^2}} \, dw &= \int_{S^{n-2}} \int_0^r \frac{r}{\sqrt{r^2 - t^2}} t^{n-2} \, dt \, d\sigma_{n-2}(\xi) \notag \\
			&= |S^{n-2}(1)| \int_0^r \frac{r t^{n-2}}{\sqrt{r^2 - t^2}} \, dt, \label{nppp2}
		\end{align}
		where \( |S^{n-2}(1)| \) represents the surface area of the unit \( (n-2) \)-sphere, which is the high-dimensional unit sphere surface. Next, we use the substitution \( t = r u \), so that \( dt = r \, du \). This leads to
		\begin{align}
			\int_{\|w\| < r} \frac{r}{\sqrt{r^2 - \|w\|^2}} \, dw &= |S^{n-2}(1)| r^{n-1} \int_0^1 \frac{u^{n-2}}{\sqrt{1 - u^2}} \, du. \label{nppp3}
		\end{align}
		We now recognize the right integral above as a standard form, and with the substitution \( s = u^2 \), \( ds = 2u \, du \), it becomes
		\begin{align}
			\int_0^1 \frac{u^{n-2}}{\sqrt{1 - u^2}} \, du &= \frac{1}{2} \int_0^1 s^{\frac{n-3}{2}} (1 - s)^{-1/2} \, ds \notag \\
			&= \frac{1}{2} B\left( \frac{n-1}{2}, \frac{1}{2} \right), \label{nppp4}
		\end{align}
		where \( B(a, b) \) is the Beta function, which is defined as
		\[
		B(a, b) = \int_0^1 s^{a-1} (1 - s)^{b-1} \, ds.
		\]
		Substituting \eqref{nppp3} and \eqref{nppp4} into \eqref{nppp2} yields
		\begin{align}
			I_{2}(r) = \frac{M}{\omega_n} \frac{|S^{n-2}(1)|}{2} B\left( \frac{n-1}{2}, \frac{1}{2} \right)=O(1), \label{nppp5}
		\end{align}
		in which $B\left( \frac{n-1}{2}, \frac{1}{2} \right)$ is a positive constant. \par
		
		Hence, it follows from \eqref{AAA}, \eqref{BBB} and \eqref{nppp5} that	
        \begin{equation} \label{E-17} m(r, f(x + c_\zeta) - f({x}))=I_{1}+I_{2}= o(T(r, f))\end{equation} 
        where \( r \) runs to infinity outside of a set of zero upper density measure \( E \).\par
		
		In the general case, any \(c \in \mathbb{R}^n\) can be written as \(c = \sum_{\zeta =0}^n c_{\zeta }\) where \( c_{0 } := (0, \cdots, 0)\). Therefore, we have
		\begin{align}\label{eq18}
			m(r, f(x+c) - f(x))  &= 	m\left(r,  \sum_{k=1}^n  f(x + \sum_{\zeta=0}^k c_{\zeta })-f(x + \sum_{\zeta=0}^{k-1} c_{\zeta }) )\right)\nonumber\\
			&	\leq \sum_{k=1}^n 	m\left(r,  ( f(x + \sum_{\zeta=0}^k c_{\zeta })-f(x + \sum_{\zeta=0}^{k-1} c_{\zeta}) )\right)\nonumber\\
			&= \sum_{k=1}^n o\left(T(r, f(x + \sum_{\zeta=0}^{k-1}  c_{\zeta })) \right),
		\end{align} for all \(r \notin E\).
		It follows from \eqref{E-17} that
		\begin{align}\label{TTo}
			\quad T(r, f(x +  c_{\zeta })) &= m(r, f(x +  c_{\zeta})) + N(r, f(x +  c_{\zeta })) \nonumber \\
			&\leq m\left(r, f(x +  c_{\zeta })-f(x)\right) + m(r, f) + N(r, f) \nonumber \\
			&= T(r, f) + o(T(r, f))
		\end{align} for all \( r \notin E \). Since \( c = \sum_{\zeta=0}^n c_{\zeta } \), it follows by repeated application of \eqref{TTo} that
		\begin{align*}
			T(r, f(x + c)) = T(r, f) + o(T(r, f))
		\end{align*}
		where \( r \notin E \).
		Thus by \eqref{eq18},
		\begin{align*}
			m(r, f(x+c) - f(x))= o(T(r, f))
		\end{align*}
		where \( r \) runs to infinity outside of a set of zero upper density measure \( E \).		
	\end{proof}	
	
Next, we obtain a tropical version of the logarithmic derivative lemma for the shift operator $qx$ under growth of zero order in several variables.\par
	
	\begin{lemma}\label{order}{\cite[Lemma B]{qdiflemma}}
		If $ T:\mathbb{R}^{+}\rightarrow\mathbb{R}^{+} $ is an increasing function such that
		\[
		\limsup_{r\rightarrow\infty}\frac{\log T(r)}{\log r}=0,
		\]
		then the set
		\[
		E:=\left\{r : T(C_{1}r) \geq C_{2}T(r)\right\}
		\]
		has logarithmic density 0, that is, \[
		\lim_{r \to \infty} \frac{1}{\log r} \int_{E \cap [1, r]} \frac{dt}{t}=0
		\] for all $ C_{1}>1 $ and $ C_{2}>1 $.
	\end{lemma}

	\begin{lemma}\label{qdiflemma}{\cite[Lemma 5.4]{qdiflemma}}
		Let \( T: \mathbb{R}^{+} \to \mathbb{R}^{+} \) be an increasing function and let \( U: \mathbb{R}^{+} \to \mathbb{R}^{+} \).
		If there exists a decreasing sequence \( \{c_{n}\}_{n\in\mathbb{N}} \) such that
		\[
		c_{n} \to 0 \quad \text{as} \quad n \to \infty
		\]
		and, for all \( n \in \mathbb{N} \), the set
		\[
		F_{n} := \{ r \geq 1 : U(r) < c_{n} T(r) \}
		\]
		has logarithmic density 1, then
		\[
		U(r) = o(T(r))
		\]
		on a set of logarithmic density $1.$
	\end{lemma}

	\begin{theorem}\label{qloga}
		Let $f: \mathbb{R}^n \to \mathbb{T}$ be a non-constant zero-order meromorphic function, and $q \in \mathbb{R}\setminus\{0\}$. Then
		\[
		m\big(r, f(qx) \oslash f(x)\big) = o\big(T(r, f)\big)
		\]
		on a set of logarithmic density $1.$
	\end{theorem}

	\begin{proof}
		For any \(x \in \mathbb{R}^n\),  we can use the radial distance \(t \geq 0\) to express it by
		$x = t \theta,$ where $\theta:=\frac{x}{\|x\|} \in S^{n-1}(1):=\{y\in\mathbb{R}^{n}: \|y\|=1\} $ means a unit direction vector. Then
		\begin{eqnarray}
			\label{E-100}
			m(r, f(qx) \oslash f(x)) &=&	m(r, f(q t\theta) \oslash f(t\theta))\nonumber\\&=&m(r, f(qt\theta) - f(t\theta)).\end{eqnarray}
		By Lemma \ref{L-3.7}, \begin{align}\label{E-101}
			m(r, f(qt\theta) - f(t\theta))=\frac{1}{\omega_n} \int_{S^{n-1}(1)}m(r, f_\theta(qt) - f_\theta(t))d\sigma(\theta).
		\end{align}
		For the one variable function $f_\theta(qt)-f_\theta(t),$ it follows from  \cite{q-dif} that
		\begin{eqnarray}\label{E-102}
			m(r, f_{\theta}(qt)-f_{\theta}(t)) &\leq& \frac{2|q-1|}{K^{2}} \left( 2T(K^{2}r + Kr, f_{\theta}) - f_{\theta}(0) \right)\nonumber\\
			&& + \frac{|q-1|}{K} \left(2T(Kr, f_\theta)-f_\theta(0) \right),
		\end{eqnarray}
		where $K \geq K_{0} > \max\{|q|, 1\}.$ Combining \eqref{E-100}, \eqref{E-101} with \eqref{E-102} and using Lemma \ref{L-10}, we can get
		\begin{align}\label{E-103}
			m(r, f(qx) \oslash f(x)) &\leq \frac{2|q-1|}{K^{2}} \left(2T(K^{2}r + Kr, f) - f(0) \right)\nonumber\\
			& + \frac{|q-1|}{K} \left( 2T(Kr, f) - f(0) \right).
		\end{align}
		
		By the hypothesis that $f$ is a non-constant zero-order meromorphic function, and Lemma \ref{order}, there exist $K_{1} > 1$ and $K_{2}>1$ such that for all $r$ 	on a set of logarithmic density $1.$
		\[
		T(Kr, f) \leq  K_{1} \cdot T(r, f) \quad \text{and} \quad T(K^2 r + Kr, f) \leq  K_{2} \cdot T(r, f).
		\]
		Hence, \eqref{E-103} gives
		\begin{eqnarray*}
		m(r, f(qx)-f(x)) &\leq& \frac{2|q-1|}{K^{1}} \left( 2 K_{2} \cdot T(r, f) - f(0) \right) \\&&+ \frac{|q-1|}{K} \left(2K_{1} \cdot T(r, f) - f(0) \right).
		\end{eqnarray*}
		If taking $K = k+1$ for $k \in \mathbb{N}^{+}$, then we obtain
\begin{eqnarray*}
		m(r, f(qx)-f(x)) &\leq& \left( \frac{4|q-1|K_2}{(k+1)^2} + \frac{2|q-1|K_1}{k+1} + o\left(\frac{1}{k+1}\right) \right) T(r, f)\\&:=&c_k T(r, f)
\end{eqnarray*}
		for all $r$ on a set of logarithmic density $1.$ Applying Lemma \ref{qdiflemma} with $U(r) = m(r, f(qx)-f(x))$, we conclude
		\[
		m(r, f(qx)-f(x)) = o(T(r, f))
		\]
		for all $r$ on a set of logarithmic density $1.$
		
	\end{proof}
	
	\section{Second main theorem with hypersurfaces}\label{Sec-5}
	
	In this section, we establish the second main theorem for tropical holomorphic maps from $\mathbb{R}^{n}$ into tropical projective space $\mathbb{TP}^{m},$ which generalizes the previous results due to Korhonen-Tohge \cite{korhonen-tohge-2016} and Cao-Zheng  \cite{hypersurfaces}.\par
	
	\subsection{Tropical projective space and holomorphic maps}\label{SS2.2}
	The tropical projective space is defined as
	$\mathbb{TP}^{m}=\mathbb{R}^{m+1}/\mathbb{R}(1, 1, \ldots, 1),$ that
	is, $\mathbb{TP}^m=\mathbb{T}^{m+1}\setminus\{0_{\mathbb{T}}\}/\sim$
	where $(a_{0}, a_{1}, \ldots, a_{m})\sim (b_{0}, b_{1}, \ldots,
	b_{m})$ if and only if
	$$(a_{0}, a_{1}, \ldots, a_{m})=\lambda\otimes (b_{0}, b_{1}, \ldots, b_{m})
	=(\lambda\otimes b_{0}, \lambda\otimes b_{1}, \ldots, \lambda\otimes b_{m})$$
	for some $\lambda\in \mathbb{R}.$ Denote by $[a_{0}: a_{1}: \cdots: a_{m}]$
	the equivalence class of $(a_{0}, a_{1}, \ldots, a_{m}).$ For instance, the one dimensional tropical projective $\mathbb{TP}:=\mathbb{TP}^{1}$
	is just  the completed max-plus semiring $\mathbb{T}\cup\{+\infty\}=\mathbb{R}\cup\{\pm\infty\}.$ The space $\mathbb{T}\mathbb{P}^{m}$ is a compact tropical variety  \cite{itenberg-mikhalki-shustin} \cite{mikhalkin-06}.\par
	
	\begin{definition}Let $f:=[f_{0}: f_{1}:\cdots: f_{m}]:\mathbb{R}^{n}\rightarrow \mathbb{TP}^{m}$ be a tropical holomorphic map, where $f_{0}, f_{1}, \ldots, f_{m}$ are tropical entire functions in $\mathbb{R}^{n}$ and do not have any roots which are common to all of them. Denote $\mathbf{f}=(f_{0}, f_{1}, \ldots, f_{m}):\mathbb{R}^n\rightarrow \mathbb{T}^{m+1}.$
		Then the map $\mathbf{f}$ is called a reduced representation of the tropical holomorphic map $f$ in $\mathbb{TP}^{m}.$\end{definition}
	
\begin{definition} Let \( f : \mathbb{R}^n \to \mathbb{TP}^m \) be a tropical holomorphic map with a reduced representation \( \mathbf{f} = (f_0, \ldots, f_m) \).  We define the tropical Cartan's characteristic function of \( f \) by
		\[
		T_{\mathbf{f}}(r) := \frac{1}{\omega_n} \int_{S^{n-1}(1)} \|f(r\theta)\|d\sigma(\theta)  - \|f(0)\|,
		\]
		where $\|f(r\theta)\| = \max \bigl\{ f_0(r\theta), \ldots, f_m(r\theta) \bigr\},$ $r\theta=x\in \mathbb{R}^n,$ $\theta\in S^{n-1}(1).$
	\end{definition}
	
	The order and hyperorder of $f$ are given by
	\begin{eqnarray*}
		\rho(f)=\limsup_{r\rightarrow\infty}\frac{\log T_{f}(r)}{\log r},
	\end{eqnarray*}and
	\begin{eqnarray*}
		\rho_{2}(f)=\limsup_{r\rightarrow\infty}\frac{\log\log T_{f}(r)}{\log r},
	\end{eqnarray*}
	respectively.

	\begin{lemma}
		If $f$ is a tropical meromorphic function from $\mathbb{R}^n\rightarrow\mathbb{TP}^1$, then we have $T_{f}(r)=T(r,f)+O(1).$
	\end{lemma}
	\begin{proof}
		Take \( F(x):=\max\{f_{0}, f_{1}\}. \) Then in the definition of \( T_\mathbf{f}(r) = T_f(r) \) can be written in the form
		\begin{align}\label{car1}
			F(x) = (f_{0} - f_{1})^+(x) + f_{1}(x)
		\end{align}
		
		for all \( x \in \mathbb{R}^n \). By applying \eqref{car1} with \( x = r \theta\), it follows that
		\begin{align}\label{car2}
			T_f(r) = m(r, f_{0} - f_{1}) +\frac{1}{\omega_n} \int_{S^{n-1}(1)} f_{1}(r\theta)d\sigma(\theta)- (f_{0} - f_{1})^+(0) - f_{1}(0).
		\end{align}

		Since
		\begin{align*}
			f_{1}(r\theta) = f_{1}^+( r \theta) - (-f_{1})^+( r \theta)
		\end{align*}
		for all \( x \in \mathbb{R} \), it follows, in particular, that
		\[
		\int_{S^{n-1}(1)} f_{1}(r\theta)d\sigma(\theta) =  \int_{S^{n-1}(1)} f_{1}^+(r\theta)d\sigma(\theta) -  \int_{S^{n-1}(1)} (-f_{1})^+(r\theta)d\sigma(\theta),
		\]
		and so
		\begin{align}\label{car3}
			T_{f_{1}}(r) = m(r, f_{1}) - m(r, 1_{\mathbb{T}} \oslash f_{1}).
		\end{align}
		
		Moreover, by the tropical Jensen formula \eqref{Jensen-1}, we have
		\begin{align}\label{car4}
			m(r, f_{1}) - m(r, 1_{\mathbb{T}} \oslash f_{1}) =  N(r, f_{1}) - N(r, 1_{\mathbb{T}} \oslash f_{1})+f_{1}(0)
		\end{align}
		
		By combining \eqref{car2} with \eqref{car3} and \eqref{car4} it follows that
		\begin{align}\label{car5}
			T_f(r) = m(r, f_{0}\oslash f_{1}) - N(r, f_{1}) + N(r, 1_{\mathbb{T}}\oslash f_{1}) - (f_{0} - f_{1})^+(0).
		\end{align}
		
		The counting function \( N(r, f_{1}) \) vanishes identically since \( f_{1} \) is entire. Furthermore, since the tropical entire functions \( f_{1} \) and \( f_{0}, \) do not have any common roots, it follows that
		\[
		N(r, 1_{\mathbb{T}}\oslash f_{1})= N(r, f_{0}, \oslash f_{1}).
		\]
		
		Hence \eqref{car5} becomes the desired formula
		\[
		T_f(r) = T(r, f_{0}, \oslash f_{1}) - (f_{0} \oslash f_{1})^+(0) = T(r, f) - f^+(0).
		\]
	\end{proof}
	
	\subsection{Tropical matrix and linear combination}\label{SS2.3}
	The operations of tropical addition $\oplus$ and tropical multiplication $\otimes$ for the $(m+1)\times(m+1)$
	matrices $A=(a_{ij})$ and $B=(b_{ij})$ are defined by $$A\oplus B=(a_{ij}\oplus b_{ij})$$ and
	$A\otimes B:=\left(\bigoplus _{k=0}^{m}a_{ik}\otimes b_{kj}\right),$ respectively. If an $(m+1)\times(m+1)$
	matrix $A$ contains at least one element different from $0_{\mathbb{T}}$ in each row, then $A$ is called regular.
	The tropical determinant $|A|_{o}$ of $A$ is defined by \begin{equation*}
		|A|_{o}=\bigoplus a_{0\pi(0)}\otimes a_{1\pi(1)}\otimes\cdots\otimes a_{n\pi(n)},
	\end{equation*}where the sum is taken over all permutations $\{\pi(0), \pi(1), \ldots, \pi(m)\}$ of the set
	$\{0, 1, \ldots, n\}.$ Note that an $(n+1)\times(n+1)$ matrix $A$ is regular if and only if $|A|_{0}\neq 0_{\mathbb{T}}.$\par

	There are two ways to define linear  dependent or independent over tropical sem-field. Tropical meromorphic functions $g_{0}, \ldots, g_{m}$ are tropical linearly dependent (respectively independent) if the max term \begin{eqnarray*} \bigoplus _{i=0}^{m} a_{i}\otimes g_{i} \end{eqnarray*} is attained at least twice.  They are called to be linearly dependent (respectively independent) in the Gondran-Minoux sense \cite{gondran-minoux-1}\cite{gondran-minoux-2} if there exist (respectively there
	do not exist) two disjoint subsets $I$ and $J$ of $K:=\{0, \ldots, m\}$ such that $I\cup J=K$ and
	\begin{eqnarray*} \bigoplus _{i\in I}a_{i}\otimes g_{i} =\bigoplus _{j\in J}a_{j}\otimes g_{j},
	\end{eqnarray*}that is,
	\begin{eqnarray*} \max_{i\in I}\{a_{i}+g_{i}\} =\max_{j\in J}\{a_{j}+g_{j}\},
	\end{eqnarray*}where the constants $a_{0}, a_{1}, \ldots, a_{m}\in\mathbb{T}$ are not all equal to $0_{\mathbb{T}}.$ Usually, linearly dependent in the Gondran-Minoux sense should be tropical linear  dependent, however, the inverse may be not true.\par
	
	If $a_{0}, \ldots, a_{m}\in\mathbb{T}$ and $f_{0}, \ldots, f_{m}$ are tropical entire functions, then
	\begin{eqnarray*}
		F=\bigoplus_{\nu=0}^{m} a_{\nu}\otimes f_{\nu}=\bigoplus_{i=1}^{j}a_{k_i}\otimes f_{k_i}
	\end{eqnarray*} is called a tropical linear combination of $f_{0}, f_{1}, \ldots, f_{m}$ over $\mathbb{T},$
	where the index set $\{k_{1}, \ldots, k_{j}\}\subset\{0, \ldots,
	m\}$ is such that $a_{k_i}\in \mathbb{R}$ for all $i\in\{1, \ldots,
	j\},$ while $a_{\nu}=0_{\mathbb{T}}$ if $\nu\not\in\{k_{1}, \ldots, k_{j}\}.$
	Note that if $f_{0}, \ldots, f_{m}$ are linearly independent in the
	sense of Gondran and Minous, then the express of $F$ cannot be
	rewritten by means of any other index set which is different from
	the set $\{k_{1}, \ldots, k_{j}\}.$ \par
	
	Let $G=\{f_{0}, \ldots, f_{m}\}(\neq \{0_{\mathbb{T}}\})$ be a set of tropical entire functions, linearly independent in
	the Gondran-Minoux sense, and denote \begin{equation*}\mathcal{L}_{G}=span<f_{0}, \ldots, f_{m}>
		=\left\{\bigoplus_{k=0}^{m}(a_{k}\otimes f_{k}): (a_{0}, \ldots, a_{m})\in\mathbb{R}^{m+1}_{\max}\right\}
	\end{equation*} to  be their linear span. The collection $G$ is called the spanning basis of $\mathcal{L}_{G}.$
	The dimension of $\mathcal{L}_{G}$ is defined by
	$$\dim(\mathcal{L}_{G})=\max\{\ell(F): F\in \mathcal{L}_{G}\setminus\{0_{\mathbb{T}}\}\},$$ where $\ell(F)$ is the
	shortest length of the representation of $F\in \mathcal{L}_{G}\setminus\{0_{\mathbb{T}}\}$ defined by
	$$\ell(F)=\min\{j\in\{1,\ldots, m+1\}: F=\bigoplus_{i=1}^{j}(a_{k_i}\otimes f_{k_i})\}$$ where
	$a_{k_i}\in\mathbb{R}$ with integers $0\leq k_1<k_2<\cdots <k_j\leq
	m.$ Note that usually the dimension of the tropical linear span
	space of $G$ may not be m+1, which is different from the classical
	linear algebraic. If $\ell(F)=m+1$ for a tropical linear combination
	$F$ of $f_{0}, \ldots, f_{m},$ then $F$ is said to be complete, that
	is, the coefficients $a_{k}$ in any expression of $F$ of the form
	$F=\bigoplus_{k=0}^{m}(a_{k}\otimes f_{k})$ must satisfy $a_{k}\in
	\mathbb{R}$ for all $k\in\{0, \ldots, m\}$ and in this case,
	$\mathcal{L}_{G}=m+1.$\par
	
	Let $G=\{f_{0}, \ldots, f_{m}\}$ be a set of tropical entire functions, linearly independent in the Gondran-Minoux
	sense, and let $Q\subset\mathcal{L}_{G}$ be a collection of tropical linear combinations of $G$ over
	$\mathbb{T}.$ The degree of degeneracy of $Q$ is defined to be $$ddg(Q):=card (\{F\in Q: \ell(F)<m+1\}).$$
	If $ddg(Q)=0,$ then we say $Q$ is non-degenerate. This means that the degree of degeneracy of a set of tropical
	linear combinations is the number of its non-complete elements. In this way the number of complete elements of
	$Q$ is the `actual dimension' of the subspace spanned by $Q,$ and thus the $ddg(Q)$ is the `codimension' of the
	subspace spanned by $Q$ (see \cite[Page 120-121]{book}). \par

	\subsection{Tropical hypersurfaces}\label{SS2.4}
	There are general definition of tropical hypersurfaces associated to a tropical Laurent polynomial \cite[Definition 3.6]{mikhalkin}. Here, we only consider positive integer topical powers. Hypersurfaces $V_{P}$ is the set of points where more than one monomial of $P$ reaches its  maximal value \cite[Proposition 3.3]{mikhalkin}.\par
	
	\begin{definition}  Consider a homogeneous tropical polynomial with degree $d$ in $m$ dimensional tropical projective space $\mathbb{TP}^{m}$ of the form
		\begin{eqnarray*} P(x)&=&\bigoplus_{I_{i}\in\mathcal{J}_{d}}c_{I{i}}\otimes x^{ I_{i}}\\&=&\bigoplus_{i_{0}+i_{1}+\ldots+i_{m}=d} c_{i_{0}, i_{1}, \ldots, i_{m}}\otimes x_{0}^{\otimes i_0}\otimes x_{1}^{\otimes i_1}\cdots \otimes x_{m}^{\otimes i_m},
		\end{eqnarray*}where $\mathcal{J}_{d}$ is the set of all $I_{i}=(i_{0}, i_{1}, \ldots, i_{m})\in\mathbb{N}_{0}^{n+1}$ with $\# I_{i}=i_{0}+i_1+\ldots+i_{m}=d.$
		The (homogeneous) \textbf{tropical hypersurface} $V_{P}$ in $\mathbb{TP}^{m}$ is the set of roots $x=(x_{0}, x_{1}, \ldots, x_{m})$ of $P(x),$  that is, the graph of $P$ is nonlinear at these points (corner locus). In particular, $V_{P}$ is called a tropical hyperplane whenever $d=1.$\end{definition}

	Set $M:=(_d^{m+d})-1.$  For any $I_{i}=(i_{0}, \ldots, i_{m})\in\mathcal{J}_d,$ $i\in\{0,1,\ldots, M\},$ denote $f^{I_{i}}:=f_0^{\otimes i_{0}}\otimes\cdots \otimes f_m^{\otimes i_{m}}.$ Then one can see that the composition function \begin{eqnarray*}P(f)&:=&P\circ f=\bigoplus_{i_{0}+i_{1}+\ldots+i_{m}=d} c_{i_{0}, i_{1}, \ldots, i_{m}}\otimes f_{0}^{\otimes i_0}\otimes f_{1}^{\otimes i_1}\cdots f_{m}^{\otimes i_m}\\
		&=&\bigoplus_{i=0}^{M} c_{I_{i}}\otimes f^{I_{i}}\end{eqnarray*} for a tropical holomorphic map $f:=[f_{0}, \ldots, f_{m}]: \mathbb{R}^n\rightarrow \mathbb{TP}^{m}$ and tropical hypersurface $V_{P}$ is a tropical algebraical combination of $f_{0}, \ldots, f_{m}$ in the Gondran-Minoux sense. From which, we may also regard $P\circ f$ as a tropical linear combination of $f^{I_{0}}, f^{I_{1}}, \ldots, f^{I_{M}}$  in the Gondran-Minoux sense. From this point of view, we introduce some definitions similarly as in Subsection \ref{SS2.3}.\par
	
	\begin{definition}\label{D1} Tropical meromorphic functions $f_{0}, \ldots, f_{m}$ are tropical algebraically dependent (respectively independent)  if and only if  $f^{I_{0}}, \ldots, f^{I_{M}}$ are tropical linearly dependent (respectively independent). They are called to be algebraically dependent (respectively independent) in the Gondran-Minoux sense if and only if  $f^{I_{0}}, \ldots, f^{I_{M}}$ are linearly dependent (respectively independent) in the Gondran-Minoux sense.\end{definition}

	\begin{definition} Let $G=\{f_{0}, \ldots, f_{m}\}(\neq \{0_{\mathbb{T}}\})$ be a set of tropical entire functions, algebraically independent in the Gondran-Minoux sense, and denote \begin{equation*}\mathcal{\hat{L}}_{G}=span<f^{I_{0}}, \ldots, f^{I_{M}}>=\left\{\bigoplus_{k=0}^{M}a_{k}\otimes f^{I_{k}}: (a_{0}, \ldots, a_{M})\in\mathbb{R}^{M+1}_{\max}\right\}
		\end{equation*} to  be their algebraic span. The collection $G$ is called the algebraic spanning basis of $\mathcal{\hat{L}}_{G}.$ The dimension of $\mathcal{\hat{L}}_{G}$ is defined by
		$$\dim(\mathcal{\hat{L}}_{G})=\max\{\hat{\ell}(F): F\in \mathcal{\hat{L}}_{G}\setminus\{0_{\mathbb{T}}\}\},$$ where $\hat{\ell}(F)$ is the shortest length of the representation of $F\in \mathcal{\hat{L}}_{G}\setminus\{0_{\mathbb{T}}\}$ defined by $$\hat{\ell}(F)=\min\{j\in\{1,\ldots, M+1\}: F=\bigoplus_{i=1}^{j}a_{k_i}\otimes f^{I_{k_i}}\}$$ where $a_{k_i}\in\mathbb{R}$ with integers $0\leq k_1<k_2<\cdots <k_j\leq M.$\end{definition}
	
	Note that usually the dimension of the tropical algebraic span space of $G$ may not be $M+1,$ which is different from the classical linear algebraic. If $\hat{\ell}(F)=M+1$ for a tropical algebraic combination $F$ of $f_{0}, \ldots, f_{m},$ then $F$ is said to be {\bf complete}, that is, the coefficients $a_{k}$ in any expression of $F$ of the form $F=\bigoplus_{k=0}^{M}a_{k}\otimes f_{k}$ must satisfy $a_{k}\in \mathbb{R}$ for all $k\in\{0, \ldots, M\}$ such that $\mathcal{\hat{L}}_{G}=M+1.$\par
	
	Choose $c\in\mathbb{R}^n\setminus\{0\}.$   For a tropical entire function $h$ on $\mathbb{R}^{n},$ denote by
	$$\overline{h}^{[0]}:=h(x),\quad \overline{h}^{[1]}:=h(x+c)=h(x\otimes c), \cdots,
	\overline{h}^{[k]}:=h(x+kc)=h(x\otimes c^{\otimes k})$$ for all $k\in\mathbb{N}.$ The tropical Casorati determinant of a tropical holomorphic map $f:\mathbb{R}^n\rightarrow\mathbb{TP}^{m}$ with a reduced representation $(f_{0}, f_{1}, \ldots, f_{m})$ is defined by
	\begin{eqnarray*}C_{o}(f):=C_{o}(f_{0}, f_{1}, \ldots, f_{m})=\bigoplus \overline{f_{0}}^{[\pi(0)]}\otimes
		\overline{f_{1}}^{[\pi(1)]}\otimes\cdots\otimes\overline{f_{m}}^{[\pi(m)]}
	\end{eqnarray*} where the sum is taken over all permutations $\{\pi(0), \ldots, \pi(m)\}$ of $\{0, \ldots, m\}.$ Furthermore, the tropical Casorati determinant $\tilde{C}(f)=C(f^{I_{0}}, \ldots, f^{I_{M}})$ is given as
	\begin{eqnarray*} \hat{C}_{o}(f)=C_{o}(f^{I_{0}}, \ldots, f^{I_{M}})=\bigoplus \overline{f^{I_{0}}}^{[\pi(0)]}\otimes \overline{f^{I_{1}}}^{[\pi(1)]}\otimes\cdots\otimes\overline{f^{I_{M}}}^{[\pi(M)]},
	\end{eqnarray*} where the sum is taken over all permutations $\{\pi(0), \ldots, \pi(M)\}$ of $\{0, 1, \ldots, M\}.$ Clearly, when $d=1,$ we have $\hat{C}_{o}(f)=C_{o}(f).$ \par

	\begin{definition}[Tropical algebraically nondegenerated]\label{alg}
		Let $f=[f_{0}: f_{1}: \ldots: f_{m}]:\mathbb{R}^n\rightarrow\mathbb{TP}^{m}$ be a tropical holomorphic map. If for any tropical hypersurface (respectively hyperplane) $V_{P}$ in $\mathbb{TP}^{m}$ defined by a homogeneous tropical polynomial $P$ in $\mathbb{R}^{m+1},$ $f(\mathbb{R}^n)$ is not a subset of $V_{P},$ then we say that $f$ is tropical algebraically (respectively linearly) nondegenerated.
	\end{definition}

	\begin{proposition}\label{P1}  If a tropical holomorphic map $f: \mathbb{R}^n\rightarrow\mathbb{TP}^{m}$ with reduced representation $f=(f_{0}, \ldots, f_{m})$ is tropical algebraically (respectively linearly) nondegenerated, then $f_{0}, \ldots, f_{m}$ are algebraically (respectively linearly) independently in the Gondran-Minoux sense.
	\end{proposition}
	
	\begin{proof} Assume that $f=(f_{0}, \ldots, f_{m})$ is tropical algebraically (respectively linearly) nondegenerated, by Definition \ref{alg}  this means that for any hypersurface (respectively hyperplane) $V_{P}$ in $\mathbb{TP}^{m}$ defined by a homogeneous tropical polynomial $P$ in $\mathbb{T}^{m+1},$ $f(\mathbb{R}^{m})\not\subset V_{P}.$ Now if $f_{0}, \ldots, f_{m}$ are algebraically (respectively linearly) dependently in the Gondran-Minoux sense, then there exist two nonempty disjoint subsets $I$ and $J$ of $K:=\{0, \ldots, M\}$ such that $I\cup J=K$ and
		\begin{eqnarray*} \bigoplus_{i\in I}a_{i}\otimes f_{0}^{\otimes i_{0}}\otimes\cdots\otimes f_{m}^{\otimes i_{m}}=\bigoplus_{j\in J}a_{j}\otimes f_{0}^{\otimes j_{0}}\otimes\cdots\otimes f_{m}^{\otimes j_{m}}
		\end{eqnarray*} where  $M=(_d^{m+d})-1,$  all $a_{i}, a_{j}\in\mathbb{T}$ and $i_{0}+\ldots+i_{m}=j_{0}+\ldots+j_{m}\in \mathcal{J}_{d}.$
		Hence, it gives a homogeneous tropical polynomial $\tilde{P}(x)=\bigoplus_{k=0}^{M} a_{k}\otimes x_{0}^{\otimes k_{0}}\otimes\cdots\otimes x_{m}^{\otimes k_{m}}$ with degree $d=k_{0}+\ldots+k_{m}\in \mathcal{J}_{d}$ such that
		\begin{eqnarray*} \tilde{P}(f)&=&\bigoplus_{k=0}^{M} a_{k}\otimes f_{0}^{\otimes k_{0}}\otimes\cdots\otimes f_{m}^{\otimes k_{m}}\\
			&=&\bigoplus_{i\in I}a_{i}\otimes f_{0}^{\otimes i_{0}}\otimes\cdots\otimes f_{m}^{\otimes i_{m}}\\&=&\bigoplus_{j\in J}a_{j}\otimes f_{0}^{\otimes j_{0}}\otimes\cdots\otimes f_{m}^{\otimes j_{m}}.
		\end{eqnarray*} This implies that
		$(f_{0}(x), \ldots, f_{m}(x))$ are points of $V_{\tilde{P}}$ for all $x\in \mathbb{R}^{n}.$ We obtain a contradiction. Hence $f_{0}, \ldots, f_{m}$ must be algebraically (respectively linearly) dependently in the Gondran-Minoux sense.
	\end{proof}
	
	\subsection{First main theorem for tropical hypersurfaces}
	
	\begin{definition}[Weil function and proximity function]  Let $f:\mathbb{R}^n\rightarrow\mathbb{TP}^{m}$ be a tropical holomorphic map,  let $V_{P}$ be a tropical hypersurface with degree $d$ defined by a homogeneous polynomial $P$ of degree $d$ and let $a$ be the vector defined by the polynomial $P.$ The proximity function $m_{f}(r, V_{P})$ of tropical holomorphic mao $f$ with respect to tropical hypersurface $V_{P}$ is defined as
		\begin{eqnarray*}m_{f}(r, V_{P}):= \frac{1}{\omega_n} \int_{S^{n-1}(1)} 	\lambda_{V_{P}}(f(r\theta))d\sigma(\theta)                \end{eqnarray*}
		where $\lambda_{V_{P}}(f(r\theta))$ means the Weil function defined by $$\lambda_{V_{P}}(f(x)):=\frac{\|f(r\theta)\|^{\otimes d}\otimes \|a\|}{P(f)(r\theta)}\oslash.$$\end{definition}
	
	Note that $P(f)$ is a tropical entire function on $\mathbb{R}^{n}$   which  doesn't have any pole. Hence by the tropical Jensen formula \eqref{Jensen-1}, we have
	\begin{align*} &N(r, 1_{\mathbb{T}}\oslash P(f))\\&=m(r, P(f))-m(r, 1_{\mathbb{T}}\oslash P(f))-P(f)(0)\\
		&=\frac{1}{\omega_n} \int_{S^{n-1}(1)} \left(P(f))^+( r\theta)  - (-P(f))^+ ( r\theta\right)  d\sigma(\theta)  -P(f)(0)\\
		&=\frac{1}{\omega_n} \int_{S^{n-1}(1)} P(f)( r\theta)    d\sigma(\theta)  -P(f)(0)\\
		&=\frac{d}{\omega_n} \int_{S^{n-1}(1)} 	\|f(r\theta)\|d\sigma(\theta)    -\frac{1}{\omega_n} \int_{S^{n-1}(1)} 	\left(d\|f(r\theta)\|-P(f)(r\theta)  \right)  d\sigma(\theta)+O(1)\\
		&=\frac{d}{\omega_n} \int_{S^{n-1}(1)} 	\|f(r\theta)\|d\sigma(\theta)    -\frac{1}{\omega_n} \int_{S^{n-1}(1)} 	\frac{\|f(r\theta)\|^{\otimes d}}{P(f)(r\theta)}\oslash d\sigma(\theta)+O(1),
	\end{align*}  which gives the following first main theorem for tropical hypersurfaces. \par
	
	\begin{theorem}[First Main Theorem for tropical hypersurfaces] \label{T4}
		If $f(\mathbb{R}^n)\not\subset V_{P},$ then we have
		\begin{eqnarray*}m_{f}(r, V_{P})+N(r, 1_{\mathbb{T}}\oslash P(f))=d T_{f}(r)+O(1).\end{eqnarray*}
	\end{theorem}
	
	\subsection{Second main theorem for tropical hypersurfaces}\label{S3}
	
	Now we give the second main theorem for tropical hypersurfaces from higher dimension.\par
	
	\begin{theorem}\label{SMT}Let $q$ and $m$ be positive integers with $q\geq m.$ Let the tropical holomorphic curve $f: \mathbb{R}^n\rightarrow\mathbb{TP}^{m}$ be tropical algebraically nondegenerated. Assume that tropical hypersurfaces $V_{P_{j}}$ are defined by homogeneous tropical polynomials $P_{j}$ $(j=1, \ldots, q)$ with degree $d_{j},$ respectively, and $d=lcd(d_{1}, \ldots, d_{q})$ (the least common number). Let $M=(_d^{m+d})-1.$
		If $\lambda=ddg (\{P_{M+2}\circ f, \ldots, P_{q}\circ f\})$ and $\limsup_{r\rightarrow\infty}\frac{\log T_{f}(r)}{r}=0,$  then \begin{eqnarray*}
			(q-M-1-\lambda)T_{f}(r)&\leq&\sum_{j=1}^{q}\frac{1}{d_{j}}N\left(r, \frac{1_{\mathbb{T}}}{P_{j}\circ f}\oslash\right)+o(T_{f}(r))\\&&-\frac{1}{d}N\left(r, \frac{1_{\mathbb{T}}}{C_{o}\left(P_{1}^{\otimes \frac{d}{d_{1}}}\circ f, \ldots, P_{M+1}^{\otimes \frac{d}{d_{M+1}}}\circ f\right)}\oslash\right)\\
			&\leq&\sum_{j=M+2}^{q}\frac{1}{d_{j}}N\left(r, \frac{1_{\mathbb{T}}}{P_{j}\circ f}\oslash\right)+o(T_{f}(r))\\
			&\leq&(q-M-1)T_{f}(r)+o(T_{f}(r))
		\end{eqnarray*}
		where $r$ approaches infinity outside an exceptional set of zero upper density measure. \end{theorem}
	
	\begin{proof} We divide two cases as follows.\par
		
		(i). We first assume that $d_{j}=d$ holds for all $j=1, \ldots, q$ and \begin{eqnarray*} P_{j}(x)=\bigoplus_{i_{j0}+\ldots+i_{jm}=d} a_{i_{j0}, \ldots, i_{jm}}\otimes x_{1}^{\otimes i_{j0}}\otimes\cdots\otimes x_{m+1}^{\otimes i_{jm}}=\bigoplus_{k=0}^{M} a_{I_{jk}}x^{I_{jk}},\end{eqnarray*} where $M:=(_d^{m+d})-1.$ Take $f^{I_i}=f_{0}^{\otimes i_{0}}\otimes\cdots\otimes f_{m}^{\otimes i_{m}}$ $(i=0, \ldots, M)$ which are still tropical entire functions on $\mathbb{R}^n.$ Since $f=(f_{0}, f_{1}, \ldots, f_{m})$ is tropical algebraically nondegenerated, it follows from Proposition \ref{P1} and Definition \ref{D1} that  $f=(f_{0}, \ldots, f_{m})$ is algebraically independent in the Gondran-Minoux sense, and thus $\tilde{F}=(f^{I_0}, f^{I_1}, \ldots, f^{I_M})$ is linearly independent in the Gondran-Minoux sense. Denote $g_{j-1}:=P_{j}\circ f$ for all $j=1,2, \ldots, q.$ By the properties of tropical Casorati determinant, it follows that
		\begin{eqnarray*}
			C_{o}(g_{0}, \ldots, g_{M})=g_{0}\otimes \overline{g_{0}}\otimes\cdots\otimes\overline{g_{0}}^{[M]}\otimes C_{o}(1_{\mathbb{T}}, g_{1}\oslash g_{0}, \ldots, g_{M}\oslash g_{0}).
		\end{eqnarray*}Set
		\begin{equation*}
			\tilde{L}:=\frac{g_{0}\otimes \overline{g_{1}}\otimes\cdots\otimes\overline{g_{M}}^{[M]}\otimes g_{M+1}\otimes\cdots\otimes g_{q-1}}{C_{o}(g_{0}, g_{1}, \ldots, g_{M})}\oslash
		\end{equation*}
		and $$\psi:=g_{M+1}\otimes\cdots\otimes g_{q-1}, $$ which satisfies
		\begin{eqnarray*}
			\psi=\tilde{L}\otimes K,
		\end{eqnarray*} where $K:=C_{o}\left(1_{\mathbb{T}}, g_{1}\oslash g_{0}, \ldots, g_{M}\oslash g_{0}\right)\otimes\left(\overline{g_{0}}\oslash\overline{g_{1}}\right)\otimes\cdots\otimes\left(\overline{g_{0}}^{[M]}\oslash\overline{g_{M}}^{[M]}\right).$ Denote $g_{\nu}$ $(M+1\leq\nu\leq q-1)$ to be
		\begin{equation*}
			g_{\nu}:=\bigoplus_{j\in S_{\nu}} b_{j\nu}\otimes f^{I_{j\nu}}(t\theta)=\max_{j\in S_{\nu}}\left\{b_{j\nu}+f^{I_{j\nu}}(t\theta)\right\}, b_{j\nu}\in \mathbb{R},
		\end{equation*} for  index sets $S_{\nu}\subset\{0, 1, \ldots, M\}$ with cardinality $\#S_{\nu}(\leq M+1).$ Then we have
      \begin{align}\label{E-a}
			&\sum_{\nu=M+1, \# S_{\nu}=M+1}^{q-1}\left(     \frac{1}{\omega_n} \int_{S^{n-1}(1)} 	g_{\nu}(r\theta)d\sigma(\theta)       \right)\\
			&=\sum_{\nu=M+1, \# S_{\nu}=M+1}^{q-1} \left(     \frac{1}{\omega_n} \int_{S^{n-1}(1)} 	\max_{j\in S_{\nu}}\left\{b_{j\nu}+f^{I_{j\nu}}(r\theta)\right\}d\sigma(\theta)       \right)\nonumber\\
			&\geq \sum_{\nu=M+1, \# S_{\nu}=M+1}^{q-1} \left(     \frac{1}{\omega_n} \int_{S^{n-1}(1)} 	\max_{j\in S_{\nu}}\left\{f^{I_{j\nu}}(r\theta)\right\}d\sigma(\theta)       \right)\nonumber\\&+\sum_{\nu=M+1, \# S_{\nu}=M+1}^{q-1}\min_{j\in S_{\nu}}\{b_{j\nu}\}\nonumber\\
			&= \sum_{\nu=M+1, \# S_{\nu}=M+1}^{q-1} \frac{1}{\omega_n} \int_{S^{n-1}(1)} \left(\max_{j_{\nu 0}+\ldots+j_{\nu n}=d}\{j_{\nu 0}f_{0}(r\theta)+\ldots+j_{\nu n}f_{n}(r\theta)\}\right)d\sigma(\theta) \nonumber\\&
			+\sum_{\nu=M+1, \# S_{\nu}=M+1}^{q-1}\min_{j\in S_{\nu}}\{b_{j\nu}\}\nonumber\\
			&\geq \sum_{\nu=M+1, \# S_{\nu}=M+1}^{q-1}\frac{d}{\omega_n} \int_{S^{n-1}(1)}\left(\max_{j=0}^{m}\{f_{j}(r\theta)\}
			\right)d\sigma(\theta)\nonumber\\&+\sum_{\nu=M+1, \# S_{\nu}=M+1}^{q-1}\min_{j\in S_{\nu}}\{b_{j\nu}\}\nonumber.\end{align}
		
		Since $\lambda=ddg (\{P_{M+2}\circ f, \ldots, P_{q}\circ f\})$ is the number of its non-complete elements, it means that there exist $q-M-1-\lambda$ complete elements in the set $\{P_{M+2}\circ f, \ldots, P_{q}\circ f\}.$ Since for fixed $\theta\in S^{n-1}(1),$ tropical entire functions $(g_{\nu})_{\theta}(t)$ $(M+1\leq\nu\leq q-1)$ can be regard as
		piecewise linear real functions on $\mathbb{R},$ so there exist $\alpha_{\nu}, \beta_{\nu}\in\mathbb{R}$ and an interval $[r_{1}, r_{2}]\subset\mathbb{R}$ containing the origin such that $r_{1}<r_{2}$ and
		\begin{equation*}
			(g_{\nu})_{\theta}(t)=(\alpha_{\nu})_{\theta}(t)+   \beta_{\nu}.
		\end{equation*}Then we get that
		\begin{eqnarray*}
			(g_{\nu})_{\theta}(0)=\beta_{\nu}=\max_{j\in S_{\nu}}\{b_{j\nu}+f^{I_{j\nu}}(0)\}.
		\end{eqnarray*} If define \begin{equation*}
			(h_{\nu})_{\theta}(t):=(\alpha_{\nu})_{\theta}(t)+\beta_{\nu}
		\end{equation*} for all $\theta$ and all $t\in \mathbb{R},$ then by the convexity of the graph of $(g_{\nu})_{\theta}$ we get that
		\begin{equation*}
			(g_{\nu})_{\theta}(t)\geq (h_{\nu})_{\theta}(t)
		\end{equation*}for all $x\in\mathbb{R}.$ Hence
		\begin{eqnarray*}
			\left(     \frac{1}{\omega_n} \int_{S^{n-1}(1)} 	g_{\nu}(t\theta)d\sigma(\theta)       \right)&\geq& \left(     \frac{1}{\omega_n} \int_{S^{n-1}(1)} 	h_{\nu}(t\theta)d\sigma(\theta) \right)\\&\geq& \left(     \frac{1}{\omega_n} \int_{S^{n-1}(1)} 	\beta_{\nu}d\sigma(\theta) \right)\\
			&=&\beta_{\nu}.
		\end{eqnarray*} This gives that
		\begin{align}\label{E-b}
			&\frac{1}{\omega_n} \int_{S^{n-1}(1)} \psi(r\theta)d\sigma(\theta)=\sum_{\nu=M+1}^{q-1}\left(    \frac{1}{\omega_n} \int_{S^{n-1}(1)}  g_{\nu}(r\theta)d\sigma(\theta)         \right)\nonumber \\
			=&  \sum_{\nu=M+1,\# S_{\nu}=M+1}^{q-1}\left(    \frac{1}{\omega_n} \int_{S^{n-1}(1)}  g_{\nu}(r\theta)d\sigma(\theta)         \right)\nonumber\\&+\sum_{\nu=M+1, \# S_{\nu}<M+1}^{q-1}\left(    \frac{1}{\omega_n} \int_{S^{n-1}(1)}  g_{\nu}(r\theta)d\sigma(\theta)         \right)\nonumber\\
			\geq&\sum_{\nu=M+1, \# S_{\nu}=M+1}^{q-1}\left(    \frac{1}{\omega_n} \int_{S^{n-1}(1)}  g_{\nu}(r\theta)d\sigma(\theta)         \right)+\sum_{\nu=M+1, \# S_{\nu}<M+1}^{q-1}\beta_{\nu}.\end{align}\par
		
		According to the definition of tropical Cartan's characteristic function,
		\begin{eqnarray*}
			T_{f}(r)+\max_{j=0, 1, \ldots, m}\{f_{j}(0)\}= \frac{1}{\omega_n} \int_{S^{n-1}(1)} 	\max_{j=0, 1, \ldots, m}\left\{f_j(r\theta)\right\}d\sigma(\theta).
		\end{eqnarray*}
		Then it follows from \eqref{E-a} that
		\begin{eqnarray}\label{E-c}
			&&\sum_{\nu=M+1, \# S_{\nu}=M+1}^{q-1}\left(    \frac{1}{\omega_n} \int_{S^{n-1}(1)}  g_{\nu}(r\theta)d\sigma(\theta)         \right)\\\nonumber
			&\geq&\sum_{\nu=M+1, \# S_{\nu}=M+1}^{q-1}d\left(T_{f}(r)+\max_{j=0}^{m}\{f_{j}(0)\}\right)+\sum_{\nu=M+1, \# S_{\nu}=M+1}^{q-1}\min_{j\in S_{\nu}}\{b_{j\nu}\}\\\nonumber
			&=&(q-M-1-\lambda)d\left(T_{f}(r)+\max_{j=0}^{m}\{f_{j}(0)\}\right)+\sum_{\nu=M+1, \# S_{\nu}=M+1}^{q-1}\min_{j\in S_{\nu}}\{b_{j\nu}\}.
		\end{eqnarray} Therefore, combining \eqref{E-b} and \eqref{E-c} gives that
		\begin{eqnarray*}
			&&\frac{1}{\omega_n} \int_{S^{n-1}(1)} \psi_{\nu}(r\theta)d\sigma(\theta)
			\\&\geq&(q-M-1-\lambda)d\left(T_{f}(r)+\max_{j=0}^{m}\{f_{j}(0)\}\right)\\&&+\sum_{\nu=M+1, \# S_{\nu}=M+1}^{q-1}\min_{j\in S_{\nu}}\{b_{j\nu}\}+\sum_{\nu=M+1, \# S_{\nu}<M+1}^{q-1}\beta_{\nu},
		\end{eqnarray*}
		which implies an inequality of characteristic function $T_{f}(r)$ as follows
		\begin{eqnarray}\label{E-d}&&
			(q-M-1-\lambda)T_{f}(r)\\\nonumber &\leq& \frac{1}{d}\left(\frac{1}{\omega_n} \int_{S^{n-1}(1)} \psi(r\theta)d\sigma(\theta)   \right)+\frac{1}{d}\sum_{\nu=M+1, \# S_{\nu}=M+1}^{q-1}\min_{j\in S_{\nu}}\{b_{j\nu}\}\\\nonumber&&+\frac{1}{d}\sum_{\nu=M+1, \# S_{\nu}<M+1}^{q-1}\beta_{\nu}+ (q-M-1-\lambda)\max_{j=0}^{n}\{f_{j}(0)\}.
		\end{eqnarray}
		
		Next we need obtain an estimation on the first term of the right side of \eqref{E-d}.  By the tropical Jensen formula \eqref{Jensen-1} and the definition of $\psi,$ we deduce that
		\begin{eqnarray*}&&
			\frac{1}{\omega_n} \int_{S^{n-1}(1)} \psi(r\theta)d\sigma(\theta)   \\&=&\frac{1}{\omega_n} \int_{S^{n-1}(1)} \tilde{L}(r\theta)d\sigma(\theta)   +\frac{1}{\omega_n} \int_{S^{n-1}(1)} K(r\theta)d\sigma(\theta)   \\
			&=&N(r, 1_{\mathbb{T}}\oslash \tilde{L})-N(r, \tilde{L})+\tilde{L}(0)+\frac{1}{\omega_n} \int_{S^{n-1}(1)} K^+(r\theta)d\sigma(\theta)\\&& -\frac{1}{\omega_n} \int_{S^{n-1}(1)} (-K)^+(r\theta)d\sigma(\theta) \\
			&=&N(r, 1_{\mathbb{T}}\oslash \tilde{L})-N(r, \tilde{L})+\tilde{L}(0)+m(r, K)-m(r, 1_{\mathbb{T}}\oslash K)\\
			&\leq&N(r, 1_{\mathbb{T}}\oslash \tilde{L})-N(r, \tilde{L})+\tilde{L}(0)+m(r, K).
		\end{eqnarray*}
		Denote \begin{equation*}
			L=\frac{g_{0}\otimes g_{1}\otimes\cdots\otimes g_{M}\otimes\cdots\otimes g_{q-1}}{C_{o}(g_{0}, g_{1}, \ldots, g_{M})}\oslash,
		\end{equation*} which gives
		\begin{equation*}
			\tilde{L}=L\otimes \frac{\overline{g_{1}}\otimes \overline{g_{2}}^{[2]}\otimes\cdots\otimes \overline{g_{M}}^{[M]}}{g_{1}\otimes g_{2}\otimes\cdots\otimes g_{M}}\oslash.
		\end{equation*}
		Then it follows from the above inequalities that
		\begin{eqnarray}\label{E-e}&&
			\frac{1}{\omega_n} \int_{S^{n-1}(1)} \psi(r\theta)d\sigma(\theta)\\\nonumber&\leq&N(r, 1_{\mathbb{T}}\oslash \tilde{L})-N(r, \tilde{L})+L(0)+\sum_{j=0}^{M}g_{j}(j\theta)-\sum_{j=0}^{M}g_{j}(0)+m(r, K).
		\end{eqnarray}\par
		
		Below,	we will estimate $m(r, K)$ by making use of the tropical version of the logarithmic derivative lemma. Since $g_{j}\oslash g_{0}$ $(j\in\{1, 2, \ldots, M\})$ are tropical meromorphic functions, we have \begin{eqnarray*}
			&&T_{g_{j}\oslash g_{0}}(r)\\&=&       	\frac{1}{\omega_n} \int_{S^{n-1}(1)}    \max\{g_{j}( r\theta), g_{0}(r\theta)\}       d\sigma(\theta)-\max\{g_{j}(0), g_{0}(0)\}\\
			&=&      	\frac{1}{\omega_n} \int_{S^{n-1}(1)}   \max_{j_{0}+\ldots+j_{m}=d}\{c_{j_0, \ldots, j_{m}}   +j_{0}f_{0}(r\theta)+\ldots+j_{m}f_{m}(r\theta)\}   d\sigma(\theta)\\&&   -\max\{g_{j}(0), g_{0}(0)\}\\
			&=&	\frac{1}{\omega_n} \int_{S^{n-1}(1)}  \max_{j_{0}+\ldots+j_{m}=d}\{j_{0}f_{0}(r\theta)+\ldots+j_{m}f_{m}(r\theta)\}   d\sigma(\theta)\\&&-\max\{g_{j}(0), g_{0}(0)\}+\max_{j_{0}+\ldots+j_{m}=d}\{c_{j_0, \ldots, j_{m}}\}   \\
			&\leq&\frac{1}{\omega_n} \int_{S^{n-1}(1)}  d\max\{f_{0}(r\theta ), f_{1}(r\theta),\ldots, f_{m}(r\theta)\}   d\sigma(\theta)\\&&-\max\{g_{j}(0), g_{0}(0)\}+\max_{j_{0}+\ldots+j_{m}=d}\{c_{j_0, \ldots, j_{m}}\}\\
			&\leq&dT_{f}(r)+d\max\{f_{0}(0), f_{1}(0), \ldots, f_{m}(0)\}-\max\{g_{j}(0), g_{0}(0)\}\\&&+\max_{j_{0}+\ldots+j_{m}=d}\{c_{j_0, \ldots, j_{m}}\}.
		\end{eqnarray*}
		This implies that \begin{eqnarray*} \limsup_{r\rightarrow \infty}\frac{\log T_{g_{j}\oslash g_{0}}(r)}{r}\leq\limsup_{r\rightarrow\infty}\frac{\log T_{f}(r) }{r}=0,\end{eqnarray*}
		and then by Lemma \ref{newloga} we get that  for any $k\in \mathbb{N},$
		\begin{eqnarray*}
			T_{\overline{g_{j}\oslash g_{0}}^{[k]}}(r)=(1+\varepsilon(r))
			T_{g_{j}\oslash g_{0}}(r)=T_{f}(r)+o(T_{f}(r))
		\end{eqnarray*} holds for all $r\not\in E$ with $\overline{dens} E=0$ (throughout this proof, the notation $E$ always means having the property $\overline{dens} E=0$). Therefore, for any $k\in \mathbb{N}$ \begin{eqnarray*}\limsup_{r\rightarrow \infty}\frac{\log T_{\overline{g_{j}\oslash g_{0}}^{[k]}}(r)}{r}\leq\limsup_{r\rightarrow\infty}\frac{\log T_{f}(r) }{r}=0.\end{eqnarray*}
		Note that
		\begin{eqnarray*}
			K&=&C_{o}(1_{\mathbb{T}}, g_{1}\oslash g_{0}, \ldots, g_{M}\oslash g_{0})\otimes(\overline{g_{0}}\oslash\overline{g_{1}})\otimes\cdots\otimes(\overline{g_{0}}^{[M]}\oslash\overline{g_{M}}^{[M]})\\
			&=&\frac{\bigoplus(\overline{g_{1}}^{[\pi(0)]}\oslash \overline{g_{0}}^{[\pi(0)]})\otimes \ldots\otimes (\overline{g_{M}}^{[\pi(M)]}\oslash \overline{g_{0}}^{[\pi(M)]})}{(\overline{g_{1}}\oslash\overline{g_{0}})\otimes\cdots\otimes(\overline{g_{M}}^{[M]}\oslash\overline{g_{0}}^{[M]})}\oslash\\
			&=&\bigoplus\left(\frac{\left(\overline{g_{1}\oslash g_{0}}\right)^{[\pi(0)]}}{\overline{g_{1}\oslash g_{0}}}\oslash\right) \otimes \ldots\otimes \left(\frac{\left(\overline{g_{M}\oslash g_{0}}\right)^{[\pi(M)]}}{\overline{g_{M}\oslash g_{0}}^{[M]}}\oslash\right)
		\end{eqnarray*}
		where the tropical sum is taken over all permutations $\{\pi(0), \ldots, \pi(M)\}$ of the set $\{0, 1, \ldots, M\}.$ Now by the tropical version of the logarithmic derivative lemma (Theorem \ref{loga}), we obtain that \begin{eqnarray}\label{E-f}
			m(r, K)=o(T_{f}(r))
		\end{eqnarray}holds for all $r\not\in E$ with $\overline{dens}E=0.$\par
		
		Therefore, it follows from \eqref{E-d}, \eqref{E-e} and \eqref{E-f} that
		\begin{eqnarray}\label{E-g}
			(q-M-1-\lambda)T_{f}(r)&\leq&\frac{1}{d}N(r, 1_{\mathbb{T}}\oslash \tilde{L})-\frac{1}{d}N(r, \tilde{L})+o(T_{f}(r))
		\end{eqnarray}for all $r\not\in E$ with $\overline{dens}E=0.$\par
		
		The next step is to estimate $N(r, 1_{\mathbb{T}}\oslash \tilde{L})$ and $N(r, \tilde{L}).$ Note that
		\begin{equation*}
			\tilde{L}=L\otimes \frac{\overline{g_{1}}\otimes \overline{g_{2}}^{[2]}\otimes\cdots\otimes \overline{g_{M}}^{[M]}}{g_{1}\otimes g_{2}\otimes\cdots\otimes g_{M}}\oslash
		\end{equation*} and that $g_{1}, \ldots, g_{M}$ are tropical entire functions. Then by the tropical Jensen formula \eqref{Jensen-1},
		\begin{eqnarray}\label{E-h}
			&&N(r, 1_{\mathbb{T}}\oslash\tilde{L})-N(r, \tilde{L})\\\nonumber
			&=& \frac{1}{\omega_n} \int_{S^{n-1}(1)}  \tilde{L}(r \theta)  d\sigma(\theta) -\tilde{L}(0)    \\\nonumber
			&=&	\frac{1}{\omega_n} \int_{S^{n-1}(1)}  L(r \theta)        d\sigma(\theta)  -L(0)\\\nonumber&&+\sum_{j=1}^{M}\left(   \frac{1}{\omega_n} \int_{S^{n-1}(1)}   \overline{g_{j}}^{[j]}(r \theta)    d\sigma(\theta)+\overline{g_{j}}^{[j]}(0)\right)\\\nonumber&&-\sum_{j=1}^{M}   \left(\frac{1}{\omega_n}g_{j}(r \theta)   d\sigma(\theta)+g_{j}(0)\right)\\\nonumber
			&=&N(r, 1_{\mathbb{T}}\oslash L)-N(r, L)+\sum_{j=1}^{M}\left(N(r, 1_{\mathbb{T}}\oslash\overline{g_{j}}^{[j]})-N(r, 1_{\mathbb{T}}\oslash g_{j})\right)\\\nonumber
			&\leq&N(r, 1_{\mathbb{T}}\oslash L)-N(r, L)+\sum_{j=1}^{M}\left(N(r+jc, 1_{\mathbb{T}}\oslash g_{j})-N(r, 1_{\mathbb{T}}\oslash g_{j})\right),
		\end{eqnarray}  where the last equality follows from the translation invariance of the counting function:
		since $\overline{g_{j}}^{[j]} = g_j(x + jc)$, the poles/zeros of $\overline{g_{j}}^{[j]}$ in $B_r$ correspond bijectively
		to those of $g_j$ in $B_{r+jc}$ with the same multiplicities.
		Using the tropical Jensen formula \eqref{Jensen-1} again, we deduce that
		\begin{eqnarray*}
			&&N(r, 1_{\mathbb{T}}\oslash g_{j})\\&=&  \frac{1}{\omega_n} \int_{S^{n-1}(1)} g_{j}( r\theta) d\sigma(\theta)  -g_{j}(0)\\
			&=& \frac{1}{\omega_n} \int_{S^{n-1}(1)} \max_{j_{0}+\ldots+j_{m}=d}\{c_{j_0, \ldots, j_{m}}+j_{0}f_{0}( r\theta)   +\ldots+j_{m}f_{m}( r\theta)\}  d\sigma(\theta) -g_{j}(0)\\
			&=& \frac{1}{\omega_n} \int_{S^{n-1}(1)}\max_{j_{0}+\ldots+j_{m}=d}\{j_{0}f_{0}( r\theta)+\ldots+j_{m}f_{m}( r\theta)\}d\sigma(\theta)\\&&-g_{j}(0)+\max_{j_{0}+\ldots+j_{m}=d}\{c_{j_0, \ldots, j_{n}}\}\\
			&\leq&\frac{d}{\omega_n} \int_{S^{n-1}(1)}\max\{f_{0}( r\theta), f_{1}( r\theta),\ldots, f_{m}( r\theta)\} d\sigma(\theta) \\&&-g_{j}(0)+\max_{j_{0}+\ldots+j_{m}=d}\{c_{j_0, \ldots, j_{m}}\}\\
			&\leq&dT_{f}(r)+d\max\{f_{0}(0), f_{1}(0), \ldots, f_{m}(0)\}-g_{j}(0)\\&&+\max_{j_{0}+\ldots+j_{m}=d}\{c_{j_0, \ldots, j_{m}}\},
		\end{eqnarray*}
		This implies \begin{eqnarray*}
			\limsup_{r\rightarrow\infty} \frac{\log N(r, 1_{\mathbb{T}}\oslash g_{j})}{r}\leq \limsup_{r\rightarrow\infty} \frac{\log T_{f}(r) }{r}=0.\end{eqnarray*}
		Hence by Lemma \ref{newloga},  \begin{eqnarray}\label{E1}&&\\\nonumber&&
			N(r+jc, 1_{\mathbb{T}}\oslash g_{j})-N(r, 1_{\mathbb{T}}\oslash g_{j})=\varepsilon(r)N(r, 1_{\mathbb{T}}\oslash g_{j})= o(T_{f}(r))
		\end{eqnarray}holds for $r\not\in E$ with $\overline{dens}E=0.$ Therefore we get from \eqref{E-h} and \eqref{E1} that  \begin{eqnarray*}
			N(r, 1_{\mathbb{T}}\oslash\tilde{L})-N(r, \tilde{L})
			\leq N(r, 1_{\mathbb{T}}\oslash L)-N(r, L)+o(T_{f}(r))
		\end{eqnarray*}holds for $r\not\in E$ with $\overline{dens}E=0.$
		Combining this with \eqref{E-g} gives
		\begin{eqnarray}\label{E-i}
			(q-M-1-\lambda)T_{f}(r)&\leq&\frac{1}{d}N(r, 1_{\mathbb{T}}\oslash L)-\frac{1}{d}N(r, L)+o(T_{f}(r)).
		\end{eqnarray}for all $r\not\in E$ with $\overline{dens}E=0.$\par
		
		Note that $g_{j}$ $(j=0, \ldots, q-1)$ and $C_{o}(g_{0}, \ldots, g_{M})$ are all tropical entire functions. Then according to the definition of $L,$ we can get from the tropical Jensen formula that
		\begin{eqnarray}\label{E-j}&&
			N(r, 1_{\mathbb{T}}\oslash L)-N(r, L)\\\nonumber&=&\frac{1}{\omega_n} \int_{S^{n-1}(1)}L(r\theta)  d\sigma(\theta)-L(0)\\\nonumber
			&=&\sum_{j=0}^{q-1}\left(\frac{1}{\omega_n} \int_{S^{n-1}(1)}g_{j}(r\theta)  d\sigma(\theta)-g_{j}(0)\right)\\\nonumber
			&&-\left(\frac{1}{\omega_n} \int_{S^{n-1}(1)}C_{o}(g_{0}, \ldots, g_{M})(r\theta)  d\sigma(\theta)-C_{o}(g_{0}, \ldots, g_{M})(0)\right)\\\nonumber
			&=&\sum_{j=0}^{q-1}N(r, 1_{\mathbb{T}}\oslash g_{j})-N(r, 1_{\mathbb{T}}\oslash C_{o}(g_{0}, \ldots, g_{M})).
		\end{eqnarray}
		Now combining \eqref{E-i} and \eqref{E-j}, we get the estimation form of the second main theorem that  \begin{eqnarray}\label{E6.4}
			&&(q-M-1-\lambda)T_{f}(r)\\\nonumber
			&\leq&\frac{1}{d}\sum_{j=0}^{q-1}N(r, 1_{\mathbb{T}}\oslash g_{j})-\frac{1}{d}N(r, 1_{\mathbb{T}}\oslash C_{o}(g_{0}, \ldots, g_{M}))+o(T_{f}(r)).
		\end{eqnarray}for all $r\not\in E$ with $\overline{dens}E=0.$\par
		
		Now we will estimate $N(r, 1_{\mathbb{T}}\oslash C_{o}(g_{0}, \ldots, g_{M})).$ According to the definition of tropical Casorati determinant, we have
		\begin{eqnarray}\label{E6.1}
			&&C_{o}(g_{0}, \ldots, g_{M})\\\nonumber
			&=&\bigoplus(\overline{g_{0}}^{[\pi(0)]}\otimes \cdots \otimes \overline{g_{M}}^{[\pi(M)]})\\\nonumber
			&=&\left\{\bigoplus\left[\left(\overline{g_{0}}^{[\pi(0)]}\otimes \cdots \otimes \overline{g_{M}}^{[\pi(M)]}\right)\oslash\left(g_{0} \otimes \cdots \otimes g_{M}\right)\right]\right\}+g_{0}\otimes \cdots \otimes g_{M}\\\nonumber
			&=&\bigoplus\left[\left(\overline{g_{0}}^{[\pi(0)]}\oslash g_{0}\right) \otimes \cdots \otimes \left(\overline{g_{M}}^{[\pi(M)]}\oslash g_{M}\right)\right]+ g_{0}\otimes \cdots \otimes g_{M},\end{eqnarray} where the sum is taken over all permutations $\{\pi(0), \ldots, \pi(M)\}$ of $\{0, 1, \ldots, M\}.$ If denote $$D:=\bigoplus\left[\left(\overline{g_{0}}^{[\pi(0)]}\oslash g_{0}\right) \otimes \cdots \otimes \left(\overline{g_{M}}^{[\pi(M)]}\oslash g_{M}\right)\right],$$ then \begin{eqnarray*} D\geq \left(\overline{g_{0}}^{[\pi(0)]}\oslash g_{0}\right) \otimes \cdots \otimes \left(\overline{g_{M}}^{[\pi(M)]}\oslash g_{M}\right)\end{eqnarray*} for any  permutation $\{\pi(0), \ldots, \pi(M)\}$ of $\{0, 1, \ldots, M\}.$ By the tropical Jensen formula \eqref{Jensen-1} and \eqref{E1}, we have
		\begin{align}\label{E6.2}
        & \frac{1}{\omega_n} \int_{S^{n-1}(1)}D( r\theta) d\sigma(\theta)-D(0)\\\nonumber
			&\geq \sum_{j=0}^{M}\left( \frac{1}{\omega_n} \int_{S^{n-1}(1)} (\overline{g_{j}}^{[\pi(j)]}\oslash g_{j})(r \theta) d\sigma(\theta)\right)-D(0)\\\nonumber
			&=\sum_{j=0}^{M}\left( \frac{1}{\omega_n} \int_{S^{n-1}(1)}\overline{g_{j}}^{[\pi(j)]}(r \theta) d\sigma(\theta)- \frac{1}{\omega_n} \int_{S^{n-1}(1)}g_{j}(r \theta) d\sigma(\theta)\right)-D(0)\\\nonumber
			&= \sum_{j=0}^{M}\left(N(r, 1_{\mathbb{T}}\oslash(\overline{g_{j}}^{[\pi(j)]}))-N(r, 1_{\mathbb{T}}\oslash g_{j})\right)-D(0)+\sum_{j=0}^{M}\left(\overline{g_{j}}^{[\pi(j)]}(0)-g_{j}(0)\right)\\\nonumber
			&= \sum_{j=0}^{M}\left(N(r+jc, 1_{\mathbb{T}}\oslash g_{j})-N(r, 1_{\mathbb{T}}\oslash g_{j})\right)-D(0)+\sum_{j=0}^{M}\left(\overline{g_{j}}^{[\pi(j)]}(0)-g_{j}(0)\right)\\\nonumber
			&=o(T_{f}(r))
		\end{align} for $r\not\in E$ with $\overline{dens}E=0.$
		Hence using the tropical Jensen formula again, it gives by \eqref{E6.1} and \eqref{E6.2} that
		\begin{align}\label{E6.3}
			&N(r, 1_{\mathbb{T}}\oslash C_{o}(g_{0},\ldots, g_{M}))\\\nonumber
			&=\frac{1}{\omega_n} \int_{S^{n-1}(1)} C_{o}(g_{0}, \ldots, g_{M})(r\theta)d\sigma(\theta)-C_{o}(g_{0},\ldots, g_{M})(0)\\\nonumber
			&=\frac{1}{\omega_n} \int_{S^{n-1}(1)}D(r\theta)d\sigma(\theta)-D(0)+\sum_{j=0}^{M}\left(\frac{1}{\omega_n} \int_{S^{n-1}(1)}g_{j}(r\theta)d\sigma(\theta)-g_{j}(0)\right)\\\nonumber
			&=\frac{1}{\omega_n} \int_{S^{n-1}(1)} D(r\theta)d\sigma(\theta)-D(0)+\sum_{j=0}^{M}N(r, 1_{\mathbb{T}}\oslash g_{j})\\\nonumber
			&\geq o(T_{f}(r))+\sum_{j=1}^{M+1}N(r, \frac{1_{\mathbb{T}}}{P_{j}\circ f}\oslash)
		\end{align}holds for $r\not\in E$ with $\overline{dens}E=0.$ Submitting \eqref{E6.3} into \eqref{E6.4} gives that
		\begin{eqnarray*}&&(q-M-1-\lambda)T_{f}(r)\\
			&\leq&\frac{1}{d}\sum_{j=0}^{q-1}N(r, 1_{\mathbb{T}}\oslash g_{j})-\frac{1}{d}N(r, 1_{\mathbb{T}}\oslash C_{o}(g_{0}, \ldots, g_{M}))+o(T_{f}(r))\\
			&\leq&\frac{1}{d}\sum_{j=M+2}^{q}N(r, \frac{1_{\mathbb{T}}}{P_{j}\circ f}\oslash)+o(T_{f}(r))\\
		\end{eqnarray*}where $r$ approaches infinity outside an exceptional set of zero upper density measure.\par
		
		(ii). We now consider general case whenever the degree of homogeneous polynomials $P_{j}$ $(j=1, \ldots, q)$ are $d_{j}$ respectively. Assume that
		\begin{eqnarray*} P_{j}(x)&=&\bigoplus_{i_{j0}+\ldots+i_{jm}=d_{j}} a_{i_{j0}, \ldots, i_{jm}}\otimes x_{1}^{\otimes i_{j0}}\otimes\cdots\otimes x_{m+1}^{\otimes i_{jm}}\\
			&=&\max_{i_{j0}+\ldots+i_{jm}=d_{j}}\{a_{i_{j0}, \ldots, i_{jm}}+ i_{j0}x_{1}+\cdots+ i_{jm}x_{m+1}\}.\end{eqnarray*}
		Then \begin{eqnarray*} P_{j}^{\otimes \frac{d}{d_{j}}}(x)&=&\frac{d}{d_{j}}P_{j}(x)\\
			&=&\frac{d}{d_{j}}\bigoplus_{i_{j0}+\ldots+i_{jm}=d_{j}} a_{i_{j0}, \ldots, i_{jm}}\otimes x_{1}^{\otimes i_{j0}}\otimes\cdots\otimes x_{m+1}^{\otimes i_{jm}}\\
			&=&\max_{i_{j0}+\ldots+i_{jm}=d_{j}}\{\frac{d}{d_{j}}a_{i_{j0}, \ldots, i_{jm}}+ \frac{d}{d_{j}}i_{j0}x_{1}+\cdots+ \frac{d}{d_{j}}i_{jm}x_{m+1}\}\\
			&=&\bigoplus_{\frac{d}{d_{j}}i_{j0}+\ldots+\frac{d}{d_{j}}i_{jm}=d
			}(\frac{d}{d_{j}}a_{i_{j0}, \ldots, i_{jm}})\otimes x_{1}^{\otimes \frac{d}{d_{j}}i_{j0}}\otimes\cdots\otimes x_{m+1}^{\otimes \frac{d}{d_{j}}i_{jm}}.
		\end{eqnarray*} Thus all $P_{j}^{\otimes \frac{d}{d_{j}}}(x)$ are of degree $d.$  Furthermore, we can see that if $x_{0}$ is a root of the tropical entire function $P_{j}\circ f$ with multiplicity $	 	\nu_{1_{\mathbf{T}} \oslash\left( P_{j}\circ f\right)}(x_{0})>0,$  then $x_{0}$ should be also a root of $P_{j}^{\otimes \frac{d}{d_{j}}}\circ f$ $(=\frac{d}{d_{j}}P_{j}\circ f)$ with multiplicity $     	\nu_{1_{\mathbf{T}} \oslash \left( P_{j}^{\otimes \frac{d}{d_{j}}}\circ f \right)}(x_{0})>0            .$ The inverse is also true. This implies that \begin{eqnarray*}N(r, \frac{1_{\mathbb{T}}}{P_{j}^{\otimes \frac{d}{d_{j}}}\circ f}\oslash)=\frac{d}{d_{j}}N(r, \frac{1_{\mathbb{T}}}{P_{j}\circ f}\oslash).\end{eqnarray*} Hence by the conclusion (i), we have
		\begin{eqnarray*}&&
			(q-M-1-\lambda)T_{f}(r)\\&\leq&\frac{1}{d}\sum_{j=1}^{q}N(r, \frac{1_{\mathbb{T}}}{P_{j}^{\otimes \frac{d}{d_{j}}}\circ f}\oslash)-\frac{1}{d}N(r, \frac{1_{\mathbb{T}}}{C_{o}(P_{1}^{\otimes \frac{d}{d_{1}}}\circ f, \ldots, P_{M+1}^{\otimes \frac{d}{d_{M+1}}}\circ f)}\oslash)\\&&+o(T_{f}(r))\\
			&\leq&\frac{1}{d}\sum_{j=M+2}^{q}N(r, \frac{1_{\mathbb{T}}}{P_{j}^{\otimes \frac{d}{d_{j}}}\circ f}\oslash)+o(T_{f}(r))\\
			&=&\sum_{j=M+2}^{q}\frac{1}{d_{j}}N(r, \frac{1_{\mathbb{T}}}{P_{j}\circ f}\oslash)+o(T_{f}(r))
		\end{eqnarray*}where $r$ approaches infinity outside an exceptional set of finite upper density measure. By the first main theorem (Theorem \ref{T4}) we have
		$$N(r, \frac{1_{\mathbb{T}}}{P_{j}\circ f}\oslash)\leq d_{j} T_{f}(r)$$ for all $j=1, \ldots, q.$ Therefore, the theorem is proved immediately.\end{proof}
	
	\begin{definition} The defect of a tropical holomorphic map $f: \mathbb{R}^{n}\rightarrow\mathbb{TP}^{m}$ intersecting a tropical hypersurface $V_{P}$ given by a tropical polynomial $P$ with degree $d$ on $\mathbb{R}^{m+1}$ is defined by
		\begin{eqnarray*}
			\delta_{f}(V_{P}):=\liminf_{r\rightarrow\infty} \frac{m_{f}(r, V_{P})}{d T_{f}(r)}=1-\limsup_{r\rightarrow\infty} \frac{N(r, \frac{1_{\mathbb{T}}}{P\circ f}\oslash)}{d T_{f}(r)}.
		\end{eqnarray*}
	\end{definition} Then by Theorem \ref{SMT}, we obtain immediately the following defect relation.\par
	
	\begin{corollary}\label{defect} Let $q$ and $n$ be positive integers with $q\geq n.$ Let the tropical holomorphic map $f: \mathbb{R}^{n}\rightarrow\mathbb{TP}^{m}$ be tropical algebraically  nondegenerated. Assume that $P_{j}$ $(j=1, \ldots, q)$ are homogeneous tropical polynomials with degree $d_{j},$ and $d$ are the least common number of $d_{1}, \ldots, d_{q}.$ Let $M=(_d^{m+d})-1.$
		If $\lambda=ddg (\{P_{M+2}\circ f, \ldots, P_{q}\circ f\})$ and $\limsup_{r\rightarrow\infty}\frac{\log T_{f}(r)}{r}=0,$  then
		\begin{equation*}\sum_{j=1}^{q}\delta_{f}(V_{P_{j}})\leq M+1+\lambda,\,\, \mbox{and}\,\, \sum_{j=M+2}^{q}\delta_{f}(V_{P_{j}})\leq \lambda.\end{equation*} In special case whenever $\lambda=0,$ we get that $\delta_{f}(V_{P_{j}})=0$ for each $j\in\{M+2, \ldots,  q\}.$
	\end{corollary}
	
	For linearly nondegenerated tropical hyperplanes in Theorem \ref{SMT}, that is, $d=d_{j}=1$ for all $j=1, 2, \ldots, q,$ and $M=(_d^{m+d})-1=m,$ we get the following corollary.\par
	
	\begin{corollary}\label{C1} Let $q$ and $m$ be positive integers with $q\geq m.$ Let the tropical holomorphic curve $f: \mathbb{R}^n\rightarrow\mathbb{TP}^{m}$ be tropical linearly nondegenerated. Assume that tropical hyperplanes $V_{P_{j}}$ are defined by tropical linear polynomials $P_{j}$ $(j=1, \ldots, q).$
		If $\lambda=ddg (\{P_{m+2}\circ f, \ldots, P_{q}\circ f\})$ and $\limsup_{r\rightarrow\infty}\frac{\log T_{f}(r)}{r}=0,$  then \begin{eqnarray*}
			(q-m-1-\lambda)T_{f}(r)&\leq&\sum_{j=1}^{q}N\left(r, \frac{1_{\mathbb{T}}}{P_{j}\circ f}\oslash\right)+o(T_{f}(r))\\&&-N\left(r, \frac{1_{\mathbb{T}}}{C_{o}\left(P_{1}\circ f, \ldots, P_{m+1}\circ f\right)}\oslash\right)\\
			&\leq&\sum_{j=m+2}^{q}N\left(r, \frac{1_{\mathbb{T}}}{P_{j}\circ f}\oslash\right)+o(T_{f}(r))\\
			&\leq&(q-m-1)T_{f}(r)+o(T_{f}(r))
		\end{eqnarray*}
		where $r$ approaches infinity outside an exceptional set of zero upper density measure. \end{corollary}
	
	We give the following second main theorem  for a sufficient condition of $\lambda=0$ into one dimensional tropical projective space.\par
	
	\begin{corollary}\label{C_1} Assume that $f: \mathbb{R}^{n}\rightarrow\mathbb{TP}$ is a nonconstant tropical meromorphic function with $\limsup_{r\rightarrow\infty}\frac{\log T_{f}(r) }{r}=0,$ and $a_{j}=[a_{j1}: a_{j0}]$ $(j=1, \ldots, q)$ are distinct values of $\mathbb{TP}$ which defining tropical polynomials $P_{j}(x)=a_{j0}\otimes  x_{0}\oplus a_{j1}\otimes  x_{1}$ on $\mathbb{R}^{2},$ respectively.  If $f\not\equiv(f\oplus a_{j})\not\equiv a_{j}$ for all $j=1,2\ldots, q,$ then \begin{eqnarray}\label{E7.2}&&(q-2)T_{f}(r)\\\nonumber&=&
			\sum_{j=1}^{q}N\left(r, \frac{1_{\mathbb{T}}}{P_{j}\circ f}\oslash\right)-N\left(r, \frac{1_{\mathbb{T}}}{C_{o}(P_{1}\circ f, P_{2}\circ f)}\oslash \right)+o\left(T_{f}(r)\right)\\\nonumber
			&=&\sum_{j=3}^{q}N\left(r, \frac{1_{\mathbb{T}}}{P_{j}\circ f}\oslash\right)+o\left(T_{f}(r)\right).\end{eqnarray}  holds as $r$ approaches infinity outside an exceptional set of zero upper density measure.
	\end{corollary}

	\begin{proof}Due to $f$ tropical meromorphic on $\mathbb{R}^{n},$ we may assume $f=f_{1}\oslash f_{0}=[f_{0}: f_{1}]$ where $f_{0}$ and $f_{1}$ are two tropical entire functions on  $\mathbb{R}^{n}$ without common tropical roots.  We claim that $f$ is tropical linearly nondegenerated. Otherwise, by Proposition \ref{P1} we know that $f_0$ and $f_1$ are linearly dependently in the Gondran-Minoux sense, this implies that there exist two nonempty sets $I, J$ with $I\cap J=\emptyset$ and $I\cup J=\{0, 1\}$  such that
		\begin{eqnarray*}
			\bigoplus_{i\in I} b_{i}\otimes  f_{i}=\bigoplus_{j\in J} b_{j}\otimes  f_{j},
		\end{eqnarray*} that is, there exist two values $b_{0}, b_{1}\in\mathbb{T}$ such that $b_{0}\otimes  f_{0}=b_{1}\otimes  f_{1}.$ This means $f=f_{1}\oslash f_{0}=b_{0}\oslash b_{1},$ which contradicts to the assumption that $f$ is nonconstant. Hence, $f$ is tropical linearly nondegenerated.\par
		
		Note that $f\not\equiv(f\oplus a_{j})\not\equiv a_{j}$ for all $j=1,2\ldots, q.$ We claim that $\lambda:=ddg (\{P_{1}\circ f, \ldots, P_{q}\circ f\})=0.$  Otherwise, $\lambda>0.$ Then there exists at leat one of $\{P_{1}\circ f, \ldots, P_{q}\circ f\},$ say $P_{k}\circ f,$ satisfying $\ell(P_{k}\circ f)<2$ by the definition of the degree of degeneracy. This implies that either $P_{k}\circ f\equiv a_{k0}\otimes  f_{0}$ or
		$P_{k}\circ f\equiv a_{k1}\otimes  f_{1}.$ Thus we have either
		$(a_{k0}+ f_{0})\oplus (a_{k1}+ f_{1})\equiv a_{k0}+ f_{0}$ or
		$(a_{k0}+ f_{0})\oplus (a_{k1}+ f_{1})\equiv a_{k1}+ f_{1},$
		which contradict either $f\oplus a_{k}\not\equiv a_{k}$ or $f\oplus a_{k}\not\equiv f$ respectively. Hence, $\lambda=0.$\par
		
		Now by Corollary \ref{C1} the conclusion of the corollary is obtained.
	\end{proof}
	
	At the end of this subsection, we consider another shift operator. Choose $q\in\mathbb{R}\setminus\{0, 1\}.$   For a tropical entire function $h$ on $\mathbb{R}^{n},$ denote by
	$$\overline{h}^{[0]}:=h(x),\quad \overline{h}^{[1]}:=h(qx), \cdots,
	\overline{h}^{[k]}:=h(q^{k}x)$$ for all $k\in\mathbb{N}.$ The tropical $q$-Casorati determinant of a tropical holomorphic map $f:\mathbb{R}^n\rightarrow\mathbb{TP}^{m}$ with a reduced representation $(f_{0}, f_{1}, \ldots, f_{m})$ is defined by
	\begin{eqnarray*}C_{q}(f):=C_{q}(f_{0}, f_{1}, \ldots, f_{m})=\bigoplus \overline{f_{0}}^{[\pi(0)]}\otimes
		\overline{f_{1}}^{[\pi(1)]}\otimes\cdots\otimes\overline{f_{m}}^{[\pi(m)]}
	\end{eqnarray*} where the sum is taken over all permutations $\{\pi(0), \ldots, \pi(m)\}$ of $\{0, \ldots, m\}.$ Furthermore, the tropical $q$-Casorati determinant $\tilde{C}_{q}(f)=C_{q}(f^{I_{0}}, \ldots, f^{I_{M}})$ is given as
	\begin{eqnarray*} \hat{C}_{q}(f)=C_{q}(f^{I_{0}}, \ldots, f^{I_{M}})=\bigoplus \overline{f^{I_{0}}}^{[\pi(0)]}\otimes \overline{f^{I_{1}}}^{[\pi(1)]}\otimes\cdots\otimes\overline{f^{I_{M}}}^{[\pi(M)]},
	\end{eqnarray*} where the sum is taken over all permutations $\{\pi(0), \ldots, \pi(M)\}$ of $\{0, 1, \ldots, M\}.$ Then by a similar discussion as in the proof of Theorem \ref{SMT} and using Theorem \ref{qloga}
	instead of Theorem \ref{loga}, we also obtain the following result. The details are omitted.\par
	
	\begin{theorem} \label{qsmt} Let $p$ and $m$ be positive integers with $q\geq m.$ Let the tropical holomorphic curve $f: \mathbb{R}^n\rightarrow\mathbb{TP}^{m}$ be tropical algebraically nondegenerated. Assume that tropical hypersurfaces $V_{P_{j}}$ are defined by homogeneous tropical polynomials $P_{j}$ $(j=1, \ldots, p)$ with degree $d_{j},$ respectively, and $d=lcd(d_{1}, \ldots, d_{p})$ (the least common number). Let $M=(_d^{m+d})-1.$
		If $\lambda=ddg (\{P_{M+2}\circ f, \ldots, P_{p}\circ f\})$ and $\limsup_{r\rightarrow\infty}\frac{\log T_{f}(r)}{\log r}=0,$  then \begin{eqnarray*}
			(p-M-1-\lambda)T_{f}(r)&\leq&\sum_{j=1}^{p}\frac{1}{d_{j}}N\left(r, \frac{1_{\mathbb{T}}}{P_{j}\circ f}\oslash\right)+o(T_{f}(r))\\&&-\frac{1}{d}N\left(r, \frac{1_{\mathbb{T}}}{C_{q}\left(P_{1}^{\otimes \frac{d}{d_{1}}}\circ f, \ldots, P_{M+1}^{\otimes \frac{d}{d_{M+1}}}\circ f\right)}\oslash\right)\\
			&\leq&\sum_{j=M+2}^{p}\frac{1}{d_{j}}N\left(r, \frac{1_{\mathbb{T}}}{P_{j}\circ f}\oslash\right)+o(T_{f}(r))\\
			&\leq&(p-M-1)T_{f}(r)+o(T_{f}(r))
		\end{eqnarray*}
		on a set of logarithmic density $1.$ \end{theorem}

	\section{Second main theorem without growth condition}\label{Sec-6}
	
	Finally, we consider the completeness condition for coefficients of one tropical homogeneous polynomial $P_{j}$ in $\mathbb{TP}^m$ and obtain the following interesting second main theorem without growth condition.\par
	
	\begin{theorem}\label{C1.4}
		Let a tropical holomorphic curve $f=[f_0, f_1, \ldots, f_n]: \mathbb{R}^{n}\to\mathbb{TP}^m$ be tropical algebraically nondegenerated  (i.e., its image not in any tropical hypersurface). If a tropical homogeneous polynomial
		\begin{eqnarray*} P(x)=\bigoplus_{i_{0}+i_{1}+\ldots+i_{m}=d} a_{i_{0}, i_{1}, \ldots, i_{m}}\otimes  x_{0}^{\otimes  i_0}\otimes  x_{1}^{\otimes  i_1}\cdots \otimes  x_{m}^{\otimes  i_m},
		\end{eqnarray*} is complete (i.e. , all $a_{i_{0}, i_{1}, \ldots, i_{m}}\in\mathbb{R}$), then
		\[
		T_f(r) =\frac{1}{d} N\left( r, \frac{1_{\mathbb{T}}}{P \circ f} \oslash \right) + O(1)
		\]
		and thus $\delta_{f}(V_{P})=0.$
	\end{theorem}

	\begin{proof}Since $P(f)$ is an entire function,  by Jensen formula \eqref{Jensen-1} we have
		\begin{align}\label{new-1}
			N\left( r, \frac{1_{\mathbb{T}}}{P \circ f} \oslash \right)=\frac{1}{\omega_{n}}\int_{S^{n-1}(1)} P(f)(r\theta)d\sigma(\theta) - P(f)(0).
		\end{align}
		Furthermore,
		\begin{align}\label{new-2}
			\min \{a_{i_{0}, i_{1}, \ldots, i_{m}}\}+ d\|f(x)\| \leq P(f(x)) \leq \max \{a_{i_{0}, i_{1}, \ldots, i_{m}}\}+ d\|f(x)\|,
		\end{align} where $\|f\|=\max\{f_{0}, \ldots, f_{m}\}.$
		This gives that $$P(f)(x) = d\|f(x)\| + O(1).$$ Then combining \eqref{new-1} and \eqref{new-2} we obtain
		$$T_{f}(r)= \frac{1}{d}N\left( r, \frac{1_{\mathbb{T}}}{P \circ f} \oslash \right) + O(1). $$
	\end{proof}
	
	Whenever $m=1$ (i. e.,  $\mathbb{T} \mathbb{P}^1$), for a nonconstant tropical meromorphic function $f$ on $\mathbb{R}^{n}$ and a complete polynomial $\mathbb{TP}^1$
	\begin{equation*}
		P(x) = \bigl(a_0 \otimes  x_0\bigr) \oplus \bigl(a_1 \otimes  x_1),\,\, (a, a_1\in  \mathbb{R}),
	\end{equation*}  it follows from Theorem \ref{C1.4}	that
	\begin{eqnarray}\label{E52}		T(r, f) = T_f(r) + O(1) = N\left( r, \frac{1_{\mathbb{T}}}{P \circ f} \oslash \right) + O(1).\end{eqnarray} This is just the result of Halonen, Korhonen and Filipuk \cite[Corollary3.7]{8}.
	
	We can find that the maximum value of $P\circ f=(a_0\otimes f_0)\oplus(a_1\otimes f_1)$ is attained at least twice in $\mathbb{TP}^{1}$ is equivalent to that $f\oplus a^*$ is attained at least twice in $\mathbb{R}\cup\{\pm\infty\}$, where $a^*=[a_1:a_0]$ is the dual of $a=[a_0:a_1]$ in tropical setting. Then by the tropical Jensen formula \eqref{Jensen-1},
	\begin{align*}
		&N\left(r,\dfrac{1_{\mathbb{T}}}{P\circ f}\oslash\right)\\
		&=	\frac{1}{\omega_{n}}\int_{S^{n-1}(1)} P(f)(r\theta)d\sigma(\theta) +O(1)\\
		&=   \frac{1}{\omega_{n}}\int_{S^{n-1}(1)} (a_{0}\otimes f_{0}(r\theta))\oplus(a_{1}\otimes f_{1}(r\theta)) \sigma(\theta) +O(1)\\
		&=	\frac{1}{\omega_{n}}\int_{S^{n-1}(1)}\left(\left(a_{1}\oslash a_{0}\right)\oplus\left(f_{1}(r\theta)\right)\oslash f_{0}(r\theta)\right)d\sigma(\theta)\\&+	\frac{1}{\omega_{n}}\int_{S^{n-1}(1)}f_0(r\theta)d\sigma(\theta)+O(1)\\
		&=	\frac{1}{\omega_{n}}\int_{S^{n-1}(1)}\left(a^*\oplus f(r\theta)\right)d\sigma(\theta) +	\frac{1}{\omega_{n}}\int_{S^{n-1}(1)}f_0(r\theta)d\sigma(\theta)+O(1)\\
		&=N\left(r,\dfrac{1_{\mathbb{T}}}{f\oplus a^*}\oslash\right)-N(r, f\oplus a^*)+N\left(r,\dfrac{1_{\mathbb{T}}}{f_{0}}\oslash\right)-N(r, f_{0})+O(1)\\
		&=N\left(r,\dfrac{1_{\mathbb{T}}}{f\oplus a^*}\oslash\right)-N(r,f\oplus a^*)+N(r, f)+O(1).	
	\end{align*}
	
	Hence, \eqref{E52} is identically equal to \begin{eqnarray}\label{E55} T(r, f)= N\left(r,\dfrac{1_{\mathbb{T}}}{f\oplus a^*}\oslash\right)-N(r,f\oplus a^*)+N(r, f)+O(1)\end{eqnarray} for any $a\in\mathbb{R}.$\par

	Clearly, if $a<L_f=\inf\{f(b): b \,\mbox{is pole of}\, f\},$ then it implies that all poles $f\oplus a$ are just poles of $f$ and thus $N(r,f\oplus a)=N(r, f).$ Hence it gives from \eqref{E55} that the following corollary improves a result of Laine-Tohge \cite[Theorem 3.44]{laine-tohge}.\par
	
	\begin{corollary}\label{CCC} If $f$ is a nonconstant tropical meromorphic function on $\mathbb{R}^{n}.$ If $q (\geq1)$ distinct values $a_{1},\ldots,a_{q}\in\mathbb{R}$ satisfying
		\begin{equation*}
			\max\{a_{1},\ldots,a_{q}\}<L_f=\inf\{f(b): b\,\mbox{is a pole of}\, f\},
		\end{equation*}
		Then
		$$ qT(r,f)=\sum_{j = 1}^{q}N(r,1_{\mathbb{T}}\oslash(f\oplus a_{j}))+O(1).$$ This yields that $$\delta_{f}(a):=1-\limsup_{r\rightarrow\infty} \frac{N(r,1_{\mathbb{T}}\oslash(f\oplus a))}{T(r, f)}=0$$ holds for each $a\in \mathbb{R}$ satisfying $a<L_f:=\inf\{f(b): b\,\mbox{is a pole of}\, f\}.$
	\end{corollary}

\end{document}